\documentclass[12pt]{article}
\usepackage{amssymb, amsmath, hyperref}
\usepackage{amsthm,fixltx2e}
\usepackage[margin=1in]{geometry}
\usepackage{enumerate}
\usepackage{bm}
\allowdisplaybreaks[1]
\numberwithin{equation}{section}
\newtheorem{thm}{Theorem}[section]
\newtheorem{cor}{Corollary}[section]

\newtheorem{defn}{Definition}[section]

\newtheorem{lemma}{Lemma}[section]
\newtheorem{observation}{Observation}[section]
\usepackage{graphicx}
\begin{document}
	\vspace{-2cm}
	\markboth{G.Kalaivani and R. Rajkumar}{New matrices for the spectral theory of mixed graphs, Part II}
	\title{\LARGE\bf New matrices for the spectral theory of mixed graphs, part II}
	\author{G. Kalaivani\footnote{e-mail: {\tt kalaivani.slg@gmail.com}}~\footnote{The author is supported by the University Grants Commission (UGC), Government of India under the fellowship UGC NET-SRF (NTA Ref. No.: 211610155238)},\ \ \
		R. Rajkumar\footnote{e-mail: {\tt rrajmaths@yahoo.co.in (Corresponding author)}}\ \\
		{\footnotesize Department of Mathematics, The Gandhigram Rural Institute (Deemed to be University),}\\ \footnotesize{Gandhigram -- 624 302, Tamil Nadu, India}\\[3mm]
		%Andrei Gagarin\footnote{e-mail: {\tt andrei.gagarin@rhul.ac.uk}}\\
		%{\footnotesize Department of Computer Science, Royal Holloway, University of London,}\\
		%{\footnotesize Egham, Surrey, TW20 0EX, UK}\vspace{1mm}\\
		%\\
		%\\
		%\thanks{}\\
	}
	%\date{arXiv:1311.1707}
	\date{}
	\maketitle
	%\doublespacing
	%%%%%%%%%%%%%%%%%%%
	\begin{abstract}
	The concept of the integrated adjacency matrix for mixed graphs was first introduced in~\cite{Kalaimatrices1}, where its spectral properties were analyzed in relation to the structural characteristics of the mixed graph. Building upon this foundation, this paper introduces the integrated Laplacian matrix, the integrated signless Laplacian matrix, and the normalized integrated Laplacian matrix for mixed graphs. We further explore how the spectra of these matrices relate to the structural properties of the mixed graph.\\ 

		
	\end{abstract}
	
	\section{Introduction}
	
All graphs (directed graphs) considered in this paper may contain multiple loops and multiple edges (multiple directed loops and multiple arcs). A mixed graph generalizes both graphs and directed graphs. That is, a mixed graph may contain multiple loops, multiple edges, multiple directed loops and  multiple arcs. 
%*****“two oppositely oriented” loops at the same vertex [11]

A fundamental problem in spectral theory  of mixed graphs  is assigning an appropriate matrix to a mixed graph and analyzing the mixed graph's properties through the eigenvalues of that matrix. To achieve this,  researchers have introduced various  matrices that capture the spectral properties of mixed graphs.

In the literature, various matrices have been associated with mixed graphs under specific constraints. For mixed graphs without multiple edges, loops, multiple arcs, directed loops, or digons (pairs of arcs with the same endpoints but opposite directions), the Hermitian adjacency matrix was independently introduced by Liu and Li~\cite{Liu2015hermitian} and Guo and Mohar~\cite{guo2017hermitian}. Additionally, several other matrices for this specific case have been explored, as seen in~\cite{abudayah2022hermitian,Adiga2016mixed,Alomari2022signless,Wang2024index,Xiong2023principal,Yu2019Hermitian,Yu2017singularity,Yu2015hermitian,Yu2023k}. For mixed graphs without directed loops, the matrix in~\cite{Bapat1999generalized} was proposed. For mixed graphs without loops, directed loops, and digons, a matrix has been defined in~\cite{guo2017hermitian}. In cases where an edge is treated as a digon, a matrix is defined in~\cite{Mohar2020new}. Lastly, for mixed graphs without digons, loops, and directed loops, a matrix is presented in~\cite{Yuan2022hermitian}.

There are, however, certain limitations: (i) the matrix defined in~\cite{Adiga2016mixed} is not symmetric, leading to eigenvalues that may include complex numbers; (ii) some matrices are only applicable to specific types of mixed graphs; and (iii) in some matrices, undirected edges are represented as digons, making it unclear from the matrix entries whether two vertices are connected by an edge or a digon, even when the index set of the matrix is provided. Furthermore, the entries of the matrix defined in~\cite{Bapat1999generalized} do not clearly convey the number of edges, loops, arcs, or the orientation of arcs. As a result, not all mixed graphs can be uniquely determined from their associated matrices. 

To overcome these limitations, the integrated adjacency matrix for mixed graphs was introduced in~\cite{Kalaimatrices1}. Using this matrix, we introduce three additional matrices for mixed graphs in this paper: the integrated Laplacian matrix, the integrated signless Laplacian matrix, and the normalized integrated Laplacian matrix. We show that (i) each loopless mixed graph can be uniquely determined from its integrated Laplacian matrix, (ii) each mixed graph can be uniquely determined from its integrated signless Laplacian matrix, and (iii) each simple mixed graph can be uniquely determined from its normalized integrated Laplacian matrix.

Since these matrices are defined using the integrated adjacency matrix, we will refer to~\cite{Kalaimatrices1} as Part~I in the following discussion. While we maintain the terminology and notation from Part~I, some key definitions will be restated here for the reader's convenience.

The rest of the paper is organized as follows:	we begin by presenting some preliminary definitions and notations in Section~\ref{S2}, which serve as the foundation for subsequent discussions. In Subsection~\ref{S3}, we provide the definition of the integrated adjacency matrix for mixed graphs~\cite{Kalaimatrices1}. Additionally, we include relevant definitions and results associated with this matrix that are utilized in later sections.

In Sections~\ref{S4},~\ref{S5} and~\ref{S6}, we introduce the integrated Laplacian matrix, the  integrated signless Laplacian matrix and the normalized integrated Laplacian matrix, respectively of a mixed graph and study the interplay between the eigenvalues of these matrices with the structural properties of mixed graphs

	\section{Preliminaries}\label{S2}
	In this section, we present some notations and definitions of graphs, directed graphs and mixed graphs. A mixed graph $G$ is an ordered triple $G=(V_G,E_G,\vec{E}_G)$, where $V_G$ is the vertex set, $E_G$ is the edge set (a multiset of unordered pairs of vertices from $V_G$), and $\vec{E}_G$ is the arc set (a multiset of ordered pairs of vertices from $V_G$).	Let $V_G=\{v_1,v_2,\ldots,v_n\}$. If there is an edge joining $v_i$ and $v_j$ in $G$, we denote it as $v_i \sim v_j$. To specify the number of edges joining $v_i$ and $v_j$, we write $v_i\overset{k}{\sim} v_j$, provided there are $k$ such edges in $G$. If $v_i\overset{l}{\sim} v_i$, this indicates that there are $l$ loops at $v_i$ in $G$. If there is an arc from $v_i$ to $v_j$ in $G$, we denote it as $v_i \rightarrow v_j$. To indicate the number of arcs from $v_i$ to $v_j$, we write $v_i\overset{k}{\rightarrow} v_j$, provided there are $k$ such arcs in $G$. If $v_i\overset{l}{\rightarrow} v_i$, this means there are $l$ directed loops at $v_i$ in $G$.	Two vertices $u$ and $v$ in $G$ are said to be \textit{adjacent} if at least one of $u\sim v$, $u\rightarrow v$ or $v\rightarrow u$ holds. A vertex $u$ and an edge $e$ in $G$ are said to be \textit{incident} if $e=\{u,v\}$ for some $v$ in $G$. A vertex $u$ and an arc $a$ in $G$ are said to be \textit{incident} if $a=(u,v)$ or $a=(v,u)$ for some $v$ in $G$. 

The \textit{undirected degree} $d(u)$ of  a vertex $u$ in $G$ is the sum of the number of edges incident with $u$ and the number of loops at $u$, where each loop at $u$ is counted twice. The \textit{out-degree} $d^+(u)$ of $u$ in $G$ is the number of arcs start at $u$ (which includes the number of directed loops at $u$). The \textit{in-degree} $d^-(u)$ of $u$ in $G$ is the number of arcs end at $u$ (which includes the number of directed loops at $u$). $l(u)$ denotes the number of loops at $u$. $e(G),a(G)$ and $l(G)$ denote the number of edges (excluding loops), arcs (including directed loops) and loops in $G$, respectively. $G$ is said to be \textit{$(r,s)$-regular} if $d(u)=r$ and $d^+(u)=s=d^-(u)$ for all $u\in V_G$. A \textit{walk} in $G$ is a sequence $W:v_{n_1},e_1,v_{n_2},e_2,...,e_{k-1},v_{n_k}$, where $v_{n_i}\in V_G$, $e_j=\{v_{n_j},v_{n_{j+1}}\}$ or $e_j=(v_{n_j},v_{n_{j+1}})$ or $e_j=(v_{n_{j+1}},v_{n_j})$ for $i=1,2,\ldots,k$, $j=1,2,\ldots,k-1$. In this case, we say that $W$ is a walk from $v_{n_1}$ to $v_{n_k}$. The total number of edges and arcs in $W$ (with each repeated edge or arc counted as many times as it appears), i.e., $k-1$ is said to be the \textit{length} of $W$. The walk $W$ is said to be \textit{closed} if $v_1=v_k$. A mixed graph $G_1$ is said to be a \textit{submixed graph} of $G$ if $V_{G_1}\subseteq V_G$, $E_{G_1}\subseteq E_G$ and $\vec{E}_{G_1}\subseteq \vec{E}_G$. %A submixed graph $G_1$ of $G$ is said to be an \textit{induced*** submixed graph} of $G$ if $G_1$ is the maximal submixed graph of $G$ with vertex set $V_{G_1}$. 
A submixed graph $G_1$ of $G$ is said to be a \textit{spanning submixed graph} of $G$ if $V_{G_1}=V_G$.

A mixed graph $G=(V_G,E_G,\vec{E}_G)$ is said to be a graph if $\vec{E}_G=\emptyset$. In this case, it is simply denoted as $G=(V_G,E_G)$.
Let $G=(V_G,E_G)$ with $V_G=\{v_1,v_2,\ldots,v_n\}$.
The \textit{adjacency matrix} of $G$, denoted by $A(G)$, is defined as follows: The rows and the columns of $A(G)$ are indexed by the vertices of $G$, and for $i,j=1,2,\ldots,n$,
$$ \textnormal{the }(v_i,v_j)\textnormal{-th entry of } A(G) =\begin{cases}
k, & \textnormal{if}~i\neq j\textnormal{ and} ~v_i\overset{k}{\sim} v_j;\\
2l, & \textnormal{if}~i=j\textnormal{ and} ~v_i\overset{l}{\sim} v_i;\\
0, & \textnormal{otherwise}.
\end{cases}$$ 
The \textit{degree matrix} of $G$, denoted by $D(G)$, is defined as $D(G)=diag(d(v_1),d(v_2),\ldots,d(v_n))$. The \textit{Laplacian matrix} of $G$, denoted by $L(G)$, is defined as $L(G)=D(G)+A(G)$. The \textit{signless Laplacian matrix} of $G$, denoted by $Q(G)$, is defined as $Q(G)=D(G)+A(G)$. The \textit{normalized Laplacian matrix} of $G$, denoted by $\widehat{L}(G)$, is defined as $n\times n$ matrix whose rows and columns are indexed by $V_G$ and 
$$ \textnormal{the }(v_i,v_j)\textnormal{-th entry of }\widehat{L}(G) =\begin{cases}
1-\frac{2l(v_i)}{d(v_i)}, & \textnormal{if}~i= j\textnormal{ and } d(v_i)\neq 0;\\
-\frac{k}{\sqrt{d(v_{i})d(v_{j})}}, & \textnormal{if}~i\neq j\textnormal{ and} ~v_i\overset{k}{\sim} v_i;\\
0, & \textnormal{otherwise}.
\end{cases}$$
In particular, if $G$ has no isolated vertices, we have $\widehat{L}(G)=D(G)^{-\frac{1}{2}}L(G)D(G)^{-\frac{1}{2}}$. 

Since $A(G),L(G),Q(G)$ and $\widehat{L}(G)$ are real symmetric matrices, we denote their eigenvalues respectively as $\lambda_1(G)\geq\lambda_2(G)\geq\cdots\geq\lambda_{n}(G)$, $\nu_1(G)\geq\nu_2(G)\geq\cdots\geq\nu_{n}(G)$, $\xi_1(G)\geq\xi_2(G)\geq\cdots\geq\xi_{n}(G)$ and $\hat{\nu}_1(G)\geq\hat{\nu}_2(G)\geq\cdots\geq\hat{\nu}_{n}(G)$. The path graph and the cycle graph on $n$ vertices are denoted by $P_n$ and $C_n$, respectively. The disjoint union of $m$ copies of a graph $G$ is denoted by $mG$. For $S\subseteq V_G$, \textit{volume of $S$} denoted by $\text{vol}(S)$ which is defined as $\text{vol}(S)=\underset{v\in S}{\sum}d(v)$. If $G$ is connected and $u,v\in V_G$ $(u\neq v$), then the \textit{distance} $d(u,v)$ between $u$ and $v$ is the length of a shortest path between $u$ and $v$ in $G$ and $d(u,u)=0$. If $G$ is connected and $X,Y\subseteq V_G$, then the \textit{distance} between $X$ and $Y$, denoted by $d(X,Y)$, defined as $d(X,Y)=\min\{d(u,v)\mid u\in X\textnormal{ and }v\in Y\}$.

%	For $S\subseteq V_G$, the \textit{volume of $S$}, denoted by $\text{vol}(S)$, is defined as $\text{vol}(S)=\underset{v\in S}{\sum}d(v)$. If $G$ is connected and $u,v\in V_G$ $(u\neq v$), then the \textit{distance} $d(u,v)$ between $u$ and $v$ is the length of a shortest path joining $u$ and $v$ in $G$, and $d(u,u)=0$. For $X,Y\subseteq V_G$, the \textit{distance} between $X$ and $Y$, denoted by $d(X,Y)$, is defined as $d(X,Y)=\min\{d(u,v)\mid u\in X\textnormal{ and }v\in Y\}$.

A mixed graph $G=(V_G,E_G,\vec{E}_G)$ is said to be a directed graph if $E_G=\emptyset$. In this case, it is simply denoted as $G=(V_G,\vec{E}_G)$. Let $G=(V_G,\vec{E}_G)$ with $V_G=\{v_1,v_2,\ldots,v_n\}$.
The \textit{adjacency matrix} of $G$, denoted by $\vec{A}(G)$, is defined as follows~\cite{bapatbook}: The rows and the columns of $\vec{A}(G)$ are indexed by the vertices of $G$, and for $i,j=1,2,\ldots,n$,
$$\textnormal{the }(v_i,v_j)\textnormal{-th entry of } \vec{A}(G) =\begin{cases}
k, & \textnormal{if}~v_i\overset{k}{\rightarrow} v_j;\\
0, & \textnormal{otherwise}.
\end{cases}$$

A \textit{simple graph} (resp. \textit{simple directed graph}) is a graph (resp. directed graph) which has no loops (resp. directed loops) and no multiple edges (resp. multiple arcs). A \textit{simple mixed graph} is a mixed graph which has no loops, directed loops, multiple edges or multiple arcs. A \textit{complete graph} is a simple graph in which every pair of distinct vertices is adjacent. A \textit{complete directed graph} is a simple directed graph in which, for every pair of distinct vertices $u$ and $v$, both $u\rightarrow v$ and $v\rightarrow u$ exist. A \textit{complete mixed graph} is a simple mixed graph in which, for every pair of distinct vertices $u$ and $v$, the following relationship hold: $u\sim v$, $u\rightarrow v$ and $v\rightarrow u$. We denote the complete graph, the complete directed graph and the complete mixed graph on $n$ vertices as $K_n$, $K^D_n$ and $K^M_n$, respectively.  An \textit{oriented graph} of a graph $G$ is the directed graph obtained from $G$ by replacing each edge in $G$ by an arc.

An \textit{independent set} in a mixed graph $G$ is a set of vertices of $G$ in which no two vertices are adjacent. A \textit{$k$-partite graph} (or \textit{$k$-partite directed graph}, or \textit{$k$-partite mixed graph}, resp.) is a graph (or directed graph, or mixed graph, resp.) whose vertices can be partitioned into $k$ independent sets. If $k=2$, then we say that it is a \textit{bipartite graph} (or \textit{bipartite directed graph}, or \textit{bipartite mixed graph}, resp.). Let $G$ be a  $k$-partite graph with vertex partition $V_G=V_1\overset{.}{\cup}V_2\overset{.}{\cup}\cdots\overset{.}{\cup}V_k$ and $|V_i|=n_i$ for $i=1,2,\ldots,k$. Then $G$ is said to be \textit{complete $k$-partite} if $u\sim v$ for every $u\in V_i$ and $v\in V_j$ with $i\neq j$. We denote it by $K_{n_1,n_2,\ldots,n_k}$. In particular, $K_{k(m)}$ is used when $n_1=n_2=\cdots=n_k=m$.

Let $G$ be a $k$-partite directed graph with vertex partition $V_G=V_1\overset{.}{\cup}V_2\overset{.}{\cup}\ldots\overset{.}{\cup}V_k$ and $|V_i|=n_i$ for $i=1,2,\ldots,k$. Then $G$ is said to be \textit{complete $k$-partite} if $u\rightarrow v$ and $v\rightarrow u$ for every $u\in V_i$ and $v\in V_j$ with $i\neq j$. We denote it by $K^D_{n_1,n_2,\ldots,n_k}$. In particular, $K^D_{k(m)}$ is used when $n_1=n_2=\cdots=n_k=m$.

Let $G$ be a $k$-partite mixed graph with vertex partition $V_G=V_1\overset{.}{\cup}V_2\overset{.}{\cup}\ldots\overset{.}{\cup}V_k$ and $|V_i|=n_i$ for $i=1,2,\ldots,k$. Then $G$ is said to be \textit{complete $k$-partite} if $u\sim v$,  $u\rightarrow v$, and $v\rightarrow u$ for every $u\in V_i$ and $v\in V_j$ with $i\neq j$. We denote it by $K^M_{n_1,n_2,\ldots,n_k}$. In particular, $K^M_{k(m)}$ is used when $n_1=n_2=\cdots=n_k=m$.

Throughout this paper, $I_n$ denotes the $n\times n$ identity matrix, $J_{n\times m}$ denotes the all-ones $n\times m$ matrix, $\mathbf{0}_{n\times m}$ denotes the all-zero $n\times m$ matrix, and $\mathbf{1}_n$ denotes the all-ones matrix of size $n\times 1$. These matrices are simply referred to as $I$, $J$, $\mathbf{0}$, and $\mathbf{1}$, when their sizes are evident from the context. The characteristic polynomial of a square matrix  $M$ is denoted as $P_M(x)$, and the spectrum of $M$ is the multi-set of all the eigenvalues of $M$. For an $ n \times n$ real matrix $ M $ with real eigenvalues, we denote its eigenvalues by $ \lambda_i(M) $ for $ i = 1,2,\dots,n $, arranged in non-increasing order as $ \lambda_1(M) \geq \lambda_2(M) \geq \dots \geq \lambda_n(M) $. If $\lambda$ is an eigenvalue of $M$ with multiplicity $m$, then we write it as $\lambda^{(m)}$.
% For a  matrix $M$, $M[i|j]$ denotes the matrix obtained by deleting the $i$-th row and the $j$-th column from $M$. We denote $M[i|i]$ simply as $M[i]$. 
Let $M=[m_{ij}]$ and $N$ be matrices of order $m\times n$ and $p\times q$, respectively. The Kronecker product of $M$ and $N$, denoted by $M\otimes N$, is the $mp\times nq$ block matrix $[m_{ij}N]$.

In the rest of the paper we consider only mixed graphs having finite number of vertices.
		
%	\subsection{Integrated adjacency matrix of a mixed graph}\label{S3}
		\subsection{{Definitions and results from Part~I}}\label{S3}
	Let $G$ be a mixed graph. The \textit{undirected part} (resp. the \textit{directed part}) of $G$ is the spanning submixed graph of $G$ whose edge set is $E_G$ and arc set is $\emptyset$ (resp. edge set is $\emptyset$ and arc set is $\vec{E}_G$). The undirected part   and the directed part of $G$ are denoted as $G_u$ and $G_d$, respectively. 
	Clearly, every mixed graph can be viewed as the union of its undirected part and directed part. 

%	\begin{defn}\normalfont\label{adjacency}
		Let $G$ be a mixed graph with $V_G=\{v_1,v_2,\ldots,v_n\}$. For $i=1,2,\ldots,n$, let $v'_i$ and $v''_i$ be two copies of $v_i$. The \textit{integrated adjacency matrix} of $G$, denoted by $\mathcal{I}(G)$, is defined as the square matrix of order $2n$ whose rows and columns are indexed by the elements of the set $\{v'_1,v'_2,\ldots,v'_n,v''_1,v''_2,\ldots,v''_n\}$.
		For $i,j=1,2,\ldots,n$,
		
		the $(v'_i,v'_j)$-th entry of $\mathcal{I}(G)=\textnormal{the }(v''_i,v''_j)$-th entry of $\mathcal{I}(G)=\begin{cases}
			2\alpha, & \text{if}~i= j\textnormal{ and}~ v_{i}\overset{\alpha}{\sim} v_{j};\\
			\beta, & \text{if}~i\neq j\textnormal{ and}~ v_{i}\overset{\beta}{\sim} v_{j};\\
			0, & \text{otherwise};
		\end{cases}$
		
		the $(v'_i,v''_j)$-th entry of $\mathcal{I}(G)=\textnormal{the }(v''_j,v'_i)$-th entry of $\mathcal{I}(G)=\begin{cases}
			\gamma, & \text{if}~ v_{i}\overset{\gamma}{\rightarrow} v_{j};\\
			0, & \text{otherwise}.
		\end{cases}$	
%	\end{defn}
	
	The spectrum of $\mathcal{I}(G)$ is called the $\mathcal{I}$-\textit{spectrum} of $G$. The eigenvalues of $\mathcal{I}(G)$ are called the $\mathcal{I}$-\textit{eigenvalues} of $G$. Since $\mathcal{I}(G)$ is real symmetric, it's eigenvalues are real and so they can be arranged in a non-increasing order. We denote them as $\bm{\lambda}_1(G),\bm{\lambda}_2(G),\ldots,\bm{\lambda}_{2n}(G)$  and without loss of generality we assume that $\bm{\lambda}_1(G)\geq\bm{\lambda}_2(G)\geq\cdots\geq\bm{\lambda}_{2n}(G)$.
	
	By arranging the rows and columns of $\mathcal{I}(G)$ in the order 
	$v'_1,v'_2,\ldots,v'_n,v''_1,v''_2,\ldots,v''_n$,
	$\mathcal{I}(G)$ can be viewed as a $2\times 2$ block matrix 
	\begin{equation}\label{adjmatrix block}
		\mathcal{I}(G)=\begin{bmatrix}
			A(G_u) & \vec{A}(G_d)\\
			\vec{A}(G_d)^T & A(G_u)
		\end{bmatrix}.
	\end{equation}
	The \emph{associated graph} of $G$, denoted by $G^A$, is a graph, which is defined as follows.
	Let $V_{G^A}=\{v'_1,v'_2,\ldots,v'_n,v''_1,v''_2,\ldots,v''_n\}$. Make adjacency among the vertices of $G^A$ using the following rules: For $i,j=1,2,\ldots,n$,
	\begin{itemize}
		\item[(i)] $v'_i\overset{k}{\sim} v'_j$ and $v''_i\overset{k}{\sim} v''_j$ in $G^A$ if and only if $v_i\overset{k}{\sim} v_j$ in $G$;
		\item[(ii)] $v'_i\overset{k}{\sim} v''_j$ in $G^A$ if and only if $v_i\overset{k}{\rightarrow} v_j$ in $G$.
	\end{itemize}

	It is clear that $A(G^A)=\mathcal{I}(G)$. So the $\mathcal{I}$-spectrum of $G$ is the same as the spectrum of $A(G^A)$. From the construction of $G^A$ it is easy to observe that if $G$ is simple, then $G^A$ is a simple graph.
	
	A mixed graph $G$ is said to be \textit{$r$-regular} if $d^+(u)=d^-(u)$ and $d(u)+d^+(u)=r$ for all $u\in V_G$. It is easy to verify that $G$ is $r$-regular if and only if $r$ is an eigenvalue of $\mathcal{I}(G)$ with corresponding eigenvector $\boldsymbol{1}_{2n}$.

%		\begin{defn}\normalfont
		A mixed graph $G$ is said to have the \textit{associated-bipartite property} (shortly, \textit{AB property}), if it has no cycle of odd length and no alternating cycle of odd length with even number of arcs.
%	\end{defn}

%	\begin{defn}\normalfont
		Let $G$ be a mixed graph with at least one arc. A walk in $G$ is said to be an \textit{alternating walk} if starting from the first arc it's subsequent arcs are in alternating directions irrespective of the edges. 
	An alternating walk $W$ in $G$ is said to be an \textit{alternating path} if it satisfies the following conditions:
		\begin{enumerate}[(i)]
			\item Each vertex in $W$ occurs at most twice.
			\item If a vertex in $W$ occurs twice, then in between them there must be an odd number of arcs.
			\item Each edge in $W$ occurs at most twice.
			\item If an edge in $W$ occurs twice, then in between them there must be an odd number of arcs. 
			\item Each arc in $W$ occurs exactly once.
		\end{enumerate}
		An alternating walk in $G$ is said to be an \textit{alternating cycle} if it satisfies the above conditions~(iii)-(v) and it's starting vertex and it's ending vertex are the same.
%	\end{defn}

%	\begin{defn}\normalfont
		A submixed graph $G'$ of a mixed graph $G$ having at least one arc is said to be a \textit{special submixed graph} of $G$, if it satisfies the following conditions.
		\begin{enumerate}[(i)]
			\item Any two distinct vertices of $G'$ are joined by an alternating walk;
			\item If a component $H$ of $G'_u$ has at least one vertex having in-arc in $G'$ and at least one vertex having out-arc in $G'$, then for each vertex $v$ of $H$, there exists a closed alternating walk in $G'$ which starts and ends at $v$ and containing odd number of arcs.
		\end{enumerate}
%	\end{defn}

%	\begin{defn}\label{ASM}\normalfont
	  A submixed graph $G'$ of a mixed graph $G$ is said to be a \textit{mixed component} of $G$ if it satisfies one of the following three conditions.
		\begin{enumerate}[(i)]
			\item $G'$ is a component of $G_u$ and each vertex of $G'$ has out-degree zero in $G$.
			\item $G'$ is a component of $G_u$ and each vertex of $G'$ has in-degree zero in $G$.
			\item $G'$ is a maximal special submixed graph of $G$.
		\end{enumerate}
%	\end{defn}

%	Let $\mathcal{S}(G)$ denotes the multiset of all mixed components of a mixed graph $G$. 
	It is shown in~\cite[Theorem~6.1]{Kalaimatrices1} that  there is a 1-1 correspondence between the multiset of all mixed components of a mixed graph $G$ and the set of all components of $G^A$. The unique component of $G^A$ associated to each mixed component $H$ of $G$, which we henceforth refer to as the \textit{corresponding component of $H$}.

%	\begin{defn}\normalfont\label{APdefinition}
	Let $G$ be a mixed graph and let $H$ be a mixed component of $G$. Let $B_H=\{v\in V_H\mid \text{there exits an alternating path in }H\text{ which contains }v\text{ twice}\}$ and $C_H=\{e\in E_H\mid \text{there exits an alternating path in }H\text{ which contains }e\text{ twice}\}$.
	$H$ is said to have the \textit{associated path property} (shortly \textit{AP property}) if $H$ is either a path or $H$ has an alternating path which contains 
	(i) each vertex $v$ of $H$ exactly twice if $v\in B_H$, otherwise exactly once,
	(ii) all the arcs of $H$,  and
	(iii) each edge $e$ of $H$ exactly twice if $e\in C_H$, otherwise exactly once. 
	%\end{defn}

%	\begin{defn}\normalfont
		A mixed graph $G$ is said to be \textit{uniconnected} if $G$ has exactly one mixed component.
%	\end{defn}

	\begin{lemma}(\cite[Lemma~6.2]{Kalaimatrices1})\label{uniconnected connected}
		A mixed graph $G$ is uniconnected if and only if $G^A$ is a connected graph.
	\end{lemma}

%%%%%%%%%%%%%%%%%%%%%%%%%%%%%%%%%%%%%%%%%%%%%%%%%%%%%%%%%%

\section{Integrated Laplacian matrix of a mixed graph}\label{S4}
In this section, we introduce the integrated Laplacian matrix of a mixed graph, explore its properties, and examine the connection between its eigenvalues and the structural characteristics of the mixed graph.
	\begin{defn}\normalfont
		Let $G$ be a mixed graph with $V_G=\{v_1,v_2,\ldots,v_n\}$. For $i=1,2,\ldots,n$, let $v'_i$ and $v''_i$ be two copies of $v_i$. We define the \textit{integrated degree matrix} of $G$, denoted by $\mathcal{I}^D(G)$, as the diagonal matrix of order $2n$, with rows and columns are indexed by the elements of the set $\{v'_1,v'_2,\ldots,v'_n,v''_1,v''_2,\ldots,v''_n\}$.
		For $i=1,2,\ldots,n$,
		the $(v'_i,v'_i)$-th entry of $\mathcal{I}^D(G)=
		d(v_i)+d^{+}(v_i)$, and
		the $(v''_i,v''_i)$-th entry of $\mathcal{I}^D(G)=
		d(v_i)+d^{-}(v_i)$.
	\end{defn}
	
	\begin{defn}\normalfont
		The \textit{integrated Laplacian matrix} of a mixed graph $G$, denoted by $\mathcal{I}^L(G)$, is defined as $\mathcal{I}^L(G)=\mathcal{I}^D(G)-\mathcal{I}(G)$.
	\end{defn}
Notice that each loopless mixed graph can be determined from its integrated Laplacian matrix. The eigenvalues of $\mathcal{I}^L(G)$ are called the $\mathcal{I}^L$-\textit{eigenvalues} of $G$. These eigenvalues are denoted by $\bm{\nu}_i(G)$ for $i=1,2,\ldots,2n$. Since $\mathcal{I}^L(G)$ is real symmetric, its eigenvalues can be arranged, without loss of generality, as $\bm{\nu}_1(G)\geq\bm{\nu}_2(G)\geq\cdots\geq\bm{\nu}_{2n}(G)$.	The spectrum of $\mathcal{I}^L(G)$ is called the $\mathcal{I}^L$-\textit{spectrum} of $G$. Notice that $\mathcal{I}^D(G)=D(G^A)$ and therefore $\mathcal{I}^L(G)=L(G^A)$. This implies that the $\mathcal{I}^L$-spectrum of $G$ is identical to the spectrum of $L(G^A)$. 
	
	Using~\eqref{adjmatrix block}, $\mathcal{I}^L(G)$ can be viewed as a $2\times 2$ block matrix
	\begin{equation}\label{laplacianblock}
		\mathcal{I}^L(G)=\begin{bmatrix}
			L(G_u)+D_1 & -\vec{A}(G_d)\\
			-\vec{A}(G_d)^T & L(G_u)+D_2
		\end{bmatrix},
	\end{equation} 
	where $D_1=diag(d^+(v_1),d^+(v_2),\ldots,d^+(v_n))$ and $D_2=diag(d^-(v_1),d^-(v_2),\ldots,d^-(v_n))$. 
	
	\subsection{Results}
	We start with the following.
	\begin{observation}\normalfont
		If $G$ is a graph, then $\mathcal{I}^L(G)=I_2\otimes L(G)$. The eigenvalues of $\mathcal{I}^L(G)$ are identical to those of $L(G)$, but each appears with twice their multiplicities.
	\end{observation}
	\begin{thm}
		Let $G$ be a simple mixed graph with $V_G=\{v_1,v_2,\ldots,v_n\}$. Then the following assertions hold:
		\begin{enumerate}[(i)]
			\item $\mathcal{I}^L(G)$ is positive semi-definite. 
			\item The rank of $\mathcal{I}^L(G)$ equals $2n-k$, where $k$ is the number of mixed components of $G$.
			\item The row sum and the column sum of $\mathcal{I}^L(G)$ are zero, so zero is an eigenvalue of $\mathcal{I}^L(G)$.
			\item $\mathcal{I}^L(G)$ is a singular matrix.
			\item The cofactor matrix of $\mathcal{I}^L(G)$ has identical entries throughout.
		\end{enumerate}
	\end{thm}
	\begin{proof}
	Part~(iii) follows directly from the definition of $\mathcal{I}^L(G)$, and part (iv) follows from part~(iii).		
%		From \cite[Lemma~4.3]{bapatbook}, we know the following properties of $L(G^A)$:  $L(G^A)$ is a positive semi-definite matrix; the rank of $L(G^A)$ equals $2n-k$, where $k$ is the number of components of $G^A$; the cofactors of any two elements of $L(G^A)$ are equal. 
		Since $\mathcal{I}^L(G)=L(G^A)$, parts~(i) and~(v) follow directly from the properties of the Laplacian matrix of a simple graph (c.f.~\cite[Lemma~4.3]{bapatbook}).
		Finally, part~(ii) follows from \cite[Corollary~6.1]{Kalaimatrices1}: ``The number of mixed components of a mixed graph $G$ is equal to the number of components of $G^A$''. 
	\end{proof}
	
	\begin{thm}\label{thm-laplacian-evsum}
		For a mixed graph $G$ on $n$ vertices, $\underset{i=1}{\overset{2n}{\sum}}\bm{\nu}_i(G)=4e(G)+2a(G)$.
	\end{thm}
	\begin{proof}
		Since $\underset{i=1}{\overset{2n}{\sum}}\bm{\nu}_i(G)=tr(\mathcal{I}^L(G))=tr(\mathcal{I}^D(G)-\mathcal{I}(G))=tr(\mathcal{I}^D(G))-tr(\mathcal{I}(G))$, we have
		\begin{eqnarray}\label{laplaciansum}
			\underset{i=1}{\overset{2n}{\sum}}\bm{\nu}_i(G)
			&=&\underset{i=1}{\overset{n}{\sum}}[d^{+}(v_i)+d(v_i)]+\underset{i=n+1}{\overset{2n}{\sum}}[d^{-}(v_{i-n})+d(v_{i-n})]\nonumber-\underset{i=1}{\overset{n}{\sum}}4l(v_i)\\
			&=&\underset{i=1}{\overset{n}{\sum}}[2d(v_i)+d^{+}(v_i)+d^{-}(v_i)]-4l(G).
		\end{eqnarray} 
		 By substituting the values  $\underset{i=1}{\overset{n}{\sum}}d(v_i)=2e(G)+2l(G)$, and
		$\underset{i=1}{\overset{n}{\sum}}d^+(v_i)=a(G)=\underset{i=1}{\overset{n}{\sum}}d^-(v_i)$ (c.f \cite[Proposition~4.1]{Kalaimatrices1}) in~\eqref{laplaciansum},  we obtain the result.
	\end{proof}
	Let $G$ be a mixed graph. For $a\in \vec{E}_G$, $G-a$ denotes the mixed graph obtained by deleting the arc $a$ from $G$.
	\begin{thm}
		Let $G$ be a simple mixed graph on $n$ vertices. Let $H=G-a$ for some $a\in \vec{E}_G$. Then $0=\bm{\nu}_{2n}(H)=\bm{\nu}_{2n}(G)\leq\bm{\nu}_{2n-1}(H)\leq\bm{\nu}_{2n-1}(G)\leq\cdots\leq\bm{\nu}_{2}(G)\leq\bm{\nu}_{1}(H)\leq\bm{\nu}_{1}(G)$.
	\end{thm}
	\begin{proof}
		Notice that $H^A=G^A-e$, where $e$ is the edge of $G^A$ corresponding to the arc $a$ in $G$. Since $\bm{\nu}_i(G)=\nu_i(G^A)$ and $\bm{\nu}_i(H)=\nu_i(H^A)$ for $i=1,2,\ldots,2n$, the result follows from an interlacing relation between the Laplacian spectra of a graph and its subgraph (c.f.~\cite[Theorem~7.1.5]{cvetkovic}).
%		``If $G$ is a simple graph on $n$ vertices and $G_1=G-e$ for some $e\in E_G$, then we have $0=\nu_{n}(G_1)=\nu_{n}(G)\leq\nu_{n-1}(G_1)\leq\nu_{n-1}(G)\leq\cdots\leq\nu_{2}(G)\leq\nu_{1}(G_1)\leq\nu_{1}(G)$''.
	\end{proof}
	
	%\begin{thm}
	%	Let $G$ be a mixed graph on $n$ vertices. Then for any real number $a$, the eigenvalues of $\mathcal{I}^L(G)+aJ_{2n}$ are $\nu_1(G)\geq\nu_2(G)\geq\cdots\geq\nu_{2n-1}(G)$ and $2na$.
	%\end{thm}
	%\begin{proof}
	%	From \cite[Theorem~4.5]{bapatbook}, the eigenvalues of $L(G^A)+aJ_{2n}$ are $\nu_1(G^A)\geq\nu_2(G^A)\geq\cdots\geq\nu_{2n-1}(G^A)$ and $2na$. Since $\mathcal{I}^L(G)=L(G^A)$, the result follows.
	%\end{proof}
	
	\begin{thm}\label{thm-lapspectrum-regular}
		The Laplacian spectrum of an $r$-regular mixed graph $G$ on $n$ vertices is given by $0=r-\bm{\lambda}_1(G)\leq r-\bm{\lambda}_2(G)\leq\cdots\leq r-\bm{\lambda}_{2n}(G)$.
	\end{thm}
	\begin{proof}
		Since $G$ is $r$-regular, we have $\mathcal{I}^L(G)=rI_{2n}-\mathcal{I}(G)$ and $\bm{\lambda}_1(G)=r$. Hence the result follows.
	\end{proof}
	
	%\begin{thm}
	%	For a graph $G$, the spectrum of $\mathcal{I}^L(G)$ is same as the spectrum of $L(G)$ with twice the multiplicities.
	%\end{thm}
	%\begin{proof}
	%	For a  graph $G$, we have $\mathcal{I}^L(G)=I_2\otimes L(G)$. So, the result follows.
	%\end{proof}
	
	%\begin{thm}
	%	Let $G$ be a mixed graph with $d^+(v)+d^-(v)+d(v)=1$ for each vertex $v$ of $G$. Then the Laplacian spectrum of $G$ are $2$ with multiplicity $2e(G)+a(G)$ and $0$ with multiplicity $2e(G)+3a(G)$.
	%\end{thm}
	%\begin{proof}
	%	By rearranging the rows and the columns of $\mathcal{I}^L(G)$, we obtain the matrix
	%	$$\begin{bmatrix}
		%		I_{2e(G)+a(G)} & \mathbf{0}\\
		%		\mathbf{0} & \mathbf{0}_{a(G)}
		%	\end{bmatrix}\otimes (2I_2-J_2).$$ From Lemma~\ref{lem-cronecker}, the eigenvalues of the above matrix are $2$ and $0$ with multiplicities $2e(G)+a(G)$ and $2e(G)+3a(G)$, respectively. 
	%\end{proof}
	
	\begin{thm}\label{laplaciancomplete}
		\begin{enumerate}[(i)]
			\item The $\mathcal{I}^L$-spectrum of $K^M_{k(m)}$ is $0^{(1)}$, $(2mk)^{(k-1)}$, $(2mk-2m)^{(2mk-k)}$.
			\item The $\mathcal{I}^L$-spectrum of $K^D_{k(m)}$ is $(2mk-2m)^{(1)}$, $0^{(1)}$, $(mk)^{(k-1)}$, $(mk-2m)^{(k-1)}$, $(mk-m)^{(2km-2k)}$.
		\end{enumerate}
	\end{thm}
	\begin{proof}
		It is clear that $K^M_{k(m)}$ and $K^D_{k(m)}$ are $2m(k-1)$ and $m(k-1)$-regular, respectively. Therefore, the results follow from Theorem~\ref{thm-lapspectrum-regular} and \cite[Theorem~5.3]{Kalaimatrices1}: 
		``The $\mathcal{I}$-spectrum of $K^M_{k(m)}$ is $(2mk-2m)^{(1)}$, $(-2m)^{(k-1)}$, $0^{(2mk-k)}$, and
		the $\mathcal{I}$-spectrum of $K^D_{k(m)}$ is $(mk-m)^{(1)}$, $(m-mk)^{(1)}$, $m^{(k-1)}$, $(-m)^{(k-1)}$, and $0^{(2km-2k)}$''.
	\end{proof}
	\begin{cor}
		\begin{itemize}
			\item[(i)] The $\mathcal{I}^L$-spectrum of $K^M_n$ is  $0^{(1)}$, $2n^{(n-1)}$, $(2n-2)^{(n)}$.
			\item[(ii)] The $\mathcal{I}^L$-spectrum of $K^D_n$ is $(2n-2)^{(1)}$, $0^{(1)}$, $n^{(n-1)}$, $(n-2)^{(n-1)}$.
		\end{itemize}
	\end{cor}
	\begin{proof}
		By taking $k=n$ and $m=1$ in Theorem~\ref{laplaciancomplete}, the result follows.
	\end{proof}
	\begin{thm}\label{orientedLaplacian}
		\begin{enumerate}[(i)]
			\item If $G$ is the oriented graph of $P_n~(n\geq2)$ with all its arcs have the same direction, then the $\mathcal{I}^L$-spectrum of $G$ is $2^{(n-1)}$, $0^{(n+1)}$.
			\item If $G$ is the oriented graph of $C_n~(n\geq3)$ with all its arcs have the same direction, then the $\mathcal{I}^L$-spectrum of $G$ is $2^{(n)}$, $0^{(n)}$.
			\item If $G$ is the oriented graph of $C_n~(n\geq4)$ with all its arcs have the alternating direction, then the $\mathcal{I}^L$-spectrum of $G$ is $0^{(n)}$, and $(2-2\cos\frac{2\pi k}{n})^{(1)}$ for $k=1,2,\ldots,n$.
		\end{enumerate}
	\end{thm}
	\begin{proof}
		\begin{enumerate}[(i)]
			\item Here $G^A$ is the union of $(n-1)K_2$, and two isolated vertices. Therefore, the spectrum of $L(G^A)$ is $2^{(n-1)}$, $0^{(n+1)}$.
			\item Here $G^A=nK_2$ and so the spectrum of $L(G^A)$ is $2^{(n)}$, $0^{(n)}$. 
			\item Here $G^A$ is the union of $C_n$, and $n$ isolated vertices. So the spectrum of $L(G^A)$ is $0^{(n)}$, and $(2-2\cos\frac{2\pi k}{n})^{(1)}$ for $k=1,2,\ldots,n$. 
		\end{enumerate}
		Since the spectrum of $\mathcal{I}^L(G)$ and $L(G^A)$ are the same, we get the results.
	\end{proof}
	\begin{thm}\label{Laplacian spectrum cycle}
		Let $G$ be a uniconnected mixed graph on $n$ vertices. If $G$ has an alternating cycle of length $2n$ that includes all vertices, all arcs, and each edge of $G$ exactly twice, then the $\mathcal{I}^L$-spectrum of $G$ is $2-2\cos\frac{\pi k}{n}$ for $k=1,2,\ldots,2n$.
	\end{thm}
	\begin{proof} 
		Notice that $G^A$ is $2$-regular. Therefore, by \cite[Lemma~6.1]{Kalaimatrices1}: ``Let $G$ be a mixed graph. Then $G$ is $r$-regular if and only if $G^A$ is $r$-regular'', it follows that $G$ is also $2$-regular.
		Under the assumptions given, it is shown in~\cite[Theorem~6.3(iii)]{Kalaimatrices1} that the $\mathcal{I}$-spectrum of $G$ is $2\cos\frac{\pi k}{n}$ for $k=1,2,\ldots,2n$. Consequently, the result follows directly from Theorem~\ref{thm-lapspectrum-regular}.		
	\end{proof}
	\subsubsection{Bounds}
	\begin{thm}
		For a simple mixed graph $G$ on $n$ vertices,
		$$\bm{\nu}_{2n-1}(G)\leq\frac{4e(G)+2a(G)}{2n-1}\leq\bm{\nu}_{1}(G).$$
	\end{thm}
	\begin{proof}
		Since $\bm{\nu}_{2n}(G)=0$, it follows from Theorem~\ref{thm-laplacian-evsum} that $\underset{i=1}{\overset{2n-1}{\sum}}\bm{\nu}_i(G)=4e(G)+2a(G)$. Additionally, since $\bm{\nu}_{2n-1}(G)\leq\bm{\nu}_i(G)\leq\bm{\nu}_1(G)$ for $i=1,2,\ldots,2n-1$, we have the inequality $(2n-1)\bm{\nu}_{2n-1}(G)\leq\underset{i=1}{\overset{2n-1}{\sum}}\bm{\nu}_i(G)\leq(2n-1)\bm{\nu}_1(G)$. Thus, the result follows.
	\end{proof}
	
	Let $G$ be a mixed graph with $V_G=\{v_1,v_2,\ldots,v_n\}$. Consider the multiset $$\{x\mid x=d(v_i)+d^+(v_i) ~\text{or}~ x=d(v_i)+d^-(v_i) ~\text{for}~i=1,2,\ldots,n\}$$ If we arrange the elements of this multiset in decreasing order, denoted as $$d_1\geq d_2\geq \cdots\geq d_{2n},$$ then we have the following:
	\begin{thm}
		For a simple mixed graph $G$ on $n$ vertices, $\underset{i=1}{\overset{k}{\sum}}\bm{\nu}_i(G)\geq\underset{i=1}{\overset{k}{\sum}}d_i$ for $k=1,2,\ldots,2n$ with equality hold when $k=2n$.
	\end{thm}
	\begin{proof}
		Notice that $d_1\geq d_2\geq \cdots\geq d_{2n}$ is the decreasing sequence of the vertex degrees of $G^A$. 
		Since $\bm{\nu}_i(G)=\nu_i(G^A)$ for $i=1,2,\ldots,2n$, 	 the result follows from \cite[Theorem~7.1.3]{cvetkovic} which states: 
%		``Let $G$ be a simple graph on $n$ vertices. If the vertex degrees of $G$ are $d_1\geq d_2\geq \cdots\geq d_{n}$, then $\underset{i=1}{\overset{k}{\sum}}\nu_i(G)\geq\underset{i=1}{\overset{k}{\sum}}d_i$ for $k=1,2,\ldots,n$ with equality hold when $k=n$''. 
	\end{proof}
	
	\begin{thm}\label{thm-Laplacian}
		Let $G$ be a simple mixed graph on $n$ vertices. Then $$\bm{\nu}_1(G)\leq 2\nu_1(G_u)+\underset{u\in V_G}{\max}~d^+(u)+\underset{u\in V_G}{\max}~d^-(u).$$
	\end{thm}
	\begin{proof}
		 We consider $\mathcal{I}^L(G)$ as mentioned in~\eqref{laplacianblock}.
		Setting $M=\mathcal{I}^L(G)$, $P=L(G_u)+D_1$, $Q=-\vec{A}(G_d)$ and $R=L(G_u)+D_2$ in \cite[Proposition~1.3.16]{cvetkovic}, we obtain $\bm{\nu}_1(G)+\bm{\nu}_{2n}(G)\leq \lambda_1(L(G_u)+D_1)+\lambda_1(L(G_u)+D_2)$.
		%, where $\lambda_1(L(G_u)+D_1)$ and $\lambda_1(L(G_u)+D_2)$ are the largest eigenvalues of $L(G_u)+D_1$ and $L(G_u)+D_2$, respectively. 
	Using \cite[Lemma~3.19]{bapatbook}: ``If $A$ and $B$ are symmetric $n\times n$ matrices with real entries, then $\lambda_1(A+B)\leq\lambda_1(A)+\lambda_1(B)$'', we derive $$\displaystyle\lambda_1(L(G_u)+D_1)\leq \nu_1(L(G_u))+\underset{i=1}{\overset{n}{\max}}~d^+(v_i)$$ and $$\displaystyle\lambda_1(L(G_u)+D_2)\leq \nu_1(L(G_u))+\underset{i=1}{\overset{n}{\max}}~d^-(v_i).$$ Since $\bm{\nu}_{2n}(G)=0$, the result follows.
	\end{proof}
	
	\begin{thm}
		Let $G$ be a mixed graph on $n$ vertices. Then $$\bm{\nu}_{2n-1}(G)\leq\displaystyle\frac{2a(G)}{n}\leq\bm{\nu}_1(G).$$
	\end{thm}
	\begin{proof}
		We consider $\mathcal{I}^L(G)$ as mentioned in \eqref{laplacianblock}.
		The sums of the entries of $L(G_u)+D_1$, $L(G_u)+D_2$ and $-\vec{A}(G_d)$ are $a(G)$, $a(G)$ and $-a(G)$, respectively. From \cite[Corollary~1.3.13]{cvetkovic}, the eigenvalues $\frac{2}{n}a(G)$ and $0$ of the matrix $\frac{1}{n}a(G)(2I_2-J_2)$  interlace the eigenvalues of $\mathcal{I}^L(G)$. This completes the proof.
	\end{proof}
	
	If $G_1$ and $G_2$ are two simple mixed graphs, then the \textit{union} of $G_1$ and $G_2$, denoted by $G_1\cup G_2$, is the mixed graph with $V_{G_1\cup G_2}=V_{G_1}\cup V_{G_2}$, $E_{G_1\cup G_2}=E_{G_1}\cup E_{G_2}$, and $\vec{E}_{G_1\cup G_2}=\vec{E}_{G_1}\cup \vec{E}_{G_2}$.
	\begin{defn}\normalfont
	A simple mixed graph $ G $ is said to be a \textit{factorization} of its two spanning submixed graphs $ G_1 $ and $ G_2 $ if $ G = G_1 \cup G_2 $, with $ E_{G_1} \cap E_{G_2} = \emptyset $ and $ \vec{E}_{G_1} \cap \vec{E}_{G_2} = \emptyset $. This is denoted as $ G = G_1 \oplus G_2 $.
	\end{defn}
	\begin{thm}\label{thm-factorization ineq}
		If $G=G_1\oplus G_2$ is a factorization of a simple mixed graph $G$, then 
		\begin{enumerate}[(i)]
			\item $\bm{\nu}_{2n-1}(G)\geq\bm{\nu}_{2n-1}(G_1)+\bm{\nu}_{2n-1}(G_2)$,
			\item $\max\{\bm{\nu}_{1}(G_1),\bm{\nu}_{1}(G_2)\}\leq\bm{\nu}_{1}(G)\leq\bm{\nu}_{1}(G_1)+\bm{\nu}_{1}(G_2)$.
		\end{enumerate} 
	\end{thm}
	\begin{proof}
		Since $G=G_1\oplus G_2$, we have $G^A=G_1^A\oplus G_2^A$.  Since $\bm{\nu}_{i}(G)=\nu_{i}(G^A)$, $\bm{\nu}_{i}(G_1)=\nu_{i}(G_1^A)$ and $\bm{\nu}_{i}(G_2)=\nu_{i}(G_2^A)$ for $i=1,2n-1$, the proof follows from \cite[Theorem~3.3]{Mohar1991laplacian}, which states: ``Let $G=G_1\oplus G_2$ be a factorization of a graph $G$. Then  $\nu_{n-1}(G)\geq\nu_{n-1}(G_1)+\nu_{n-1}(G_2)$ and $\max\{\nu_{1}(G_1),\nu_{1}(G_2)\}\leq\nu_{1}(G)\leq\nu_{1}(G_1)+\nu_{1}(G_2)$''.
	\end{proof}

	\begin{cor}\label{cor-spanning ineq}
		Let $G$ be a simple mixed graph on $n$ vertices. If $H$ is a spanning submixed graph of $G$, then $\bm{\nu}_{2n-1}(H)\leq\bm{\nu}_{2n-1}(G)$ and $\bm{\nu}_{1}(H)\leq\bm{\nu}_{1}(G)$.
	\end{cor}
	\begin{proof}
		Let $G=H\oplus H'$, where $H'$ is the spanning submixed graph of $G$ with $E_{H'}=E_G\setminus E_H$ and $\vec{E}_{H'}=\vec{E}_G\setminus\vec{E}_H$. Then by Theorem~\ref{thm-factorization ineq}, we have $\bm{\nu}_{2n-1}(H)+\bm{\nu}_{2n-1}(H')\leq\bm{\nu}_{2n-1}(G)$ and $\max\{\bm{\nu}_{1}(H),\bm{\nu}_{1}(H')\}\leq\bm{\nu}_{1}(G)$. Since $\bm{\nu}_{2n-1}(H)\leq\bm{\nu}_{2n-1}(H)+\bm{\nu}_{2n-1}(H')$ and $\bm{\nu}_{1}(H)\leq\max\{\bm{\nu}_{1}(H),\bm{\nu}_{1}(H')\}$, the result follows.
	\end{proof}
	\begin{cor}
		If $G_1$ and $G_2$ are two simple mixed graphs having the same vertex set, then \\
		$\max\{\bm{\nu}_{1}(G_1),\bm{\nu}_{1}(G_2)\}\leq\bm{\nu}_{1}(G_1\cup G_2)\leq\bm{\nu}_{1}(G_1)+\bm{\nu}_{1}(G_2)$.
	\end{cor}
	\begin{proof}
		We can write $G_1\cup G_2=G_1\oplus G'_1$, where $G'_1$ is the spanning submixed graph of $G_2$ with $E_{G'_1}=E_{G_2}\setminus E_{G_1}$ and $\vec{E}_{G'_1}=\vec{E}_{G_2}\setminus \vec{E}_{G_1}$.	
		Since $G_1$ and $G_2$ are spanning submixed graphs of $G_1\cup G_2$, from Corollary~\ref{cor-spanning ineq}, we have $\bm{\nu}_{1}(G_1)\leq\bm{\nu}_{1}(G_1\cup G_2)$ and $\bm{\nu}_{1}(G_2)\leq\bm{\nu}_{1}(G_1\cup G_2)$. So, the first inequality of the result holds. By Theorem~\ref{thm-factorization ineq}, we have $\bm{\nu}_{1}(G_1\cup G_2)\leq\bm{\nu}_{1}(G_1)+\bm{\nu}_{1}(G'_1)\leq\bm{\nu}_{1}(G_1)+\bm{\nu}_{1}(G_2)$. Thus, the second inequality of the result holds.
	\end{proof}
	Let $G$ be a mixed graph and let $u,v\in V_G$. If there is no edge joining $u$ and $v$ in $G$, then we denote this by $u\not\sim v$. Similarly, if there is no arc from $u$ to $v$ in $G$, then we denote this by $u\not\rightarrow v$.
	\begin{thm}
		Let $G$ be a simple mixed graph on $n$ vertices, and let $u,v\in V_G$. Then we have the following.
		\begin{enumerate}[(i)]
			\item If $u\not\sim v$ in $G$, then $\bm{\nu}_{2n-1}(G)\leq\frac{1}{2}(d(u)+d(v)+d^+(u)+d^+(v))$ and $\bm{\nu}_{2n-1}(G)\leq\frac{1}{2}(d(u)+d(v)+d^-(u)+d^-(v))$.
			\item If $u\not\rightarrow v$ in $G$, then $\bm{\nu}_{2n-1}(G)\leq\frac{1}{2}(d(u)+d(v)+d^+(u)+d^-(v))$.
		\end{enumerate}
	\end{thm}
	\begin{proof}
		\begin{enumerate}[(i)]
			\item If $u\not\sim v$ in $G$, then $u'\not\sim v'$ and $u''\not\sim v''$ in $G^A$, where $u'$ and $u''$ (resp. $v'$ and $v''$) are the vertices of $G^A$ corresponding to the vertex $u$ (resp.  $v$) of $G$. By \cite[Theorem~7.4.4]{cvetkovic} 
%			``If $u$ and $v$ are two non-adjacent vertices of a simple graph $G$ on $n$ vertices, then $\nu_{n-1}(G)\leq\frac{1}{2}(d(u)+d(v))$''
			, we have $\nu_{2n-1}(G^A)\leq\frac{1}{2}(d(u')+d(v'))$ and $\nu_{2n-1}(G^A)\leq\frac{1}{2}(d(u'')+d(v''))$. Since $\nu_{2n-1}(G^A)=\bm{\nu}_{2n-1}(G)$ and $d(u')=d(u)+d^+(u)$, $d(v')=d(v)+d^+(v)$, $d(u'')=d(u)+d^-(u)$ and $d(v'')=d(v)+d^-(v)$, the result follows.
			\item If $u\not\rightarrow v$ in $G$, then $u'\not\sim v''$ in $G^A$. From \cite[Theorem~7.4.4]{cvetkovic}, we have $\nu_{2n-1}(G^A)\leq\frac{1}{2}(d(u')+d(v''))$. Since $\nu_{2n-1}(G^A)=\bm{\nu}_{2n-1}(G)$ and $d(u')=d(u)+d^+(u)$ and $d(v'')=d(v)+d^-(v)$, the result follows.
		\end{enumerate}
	\end{proof}
	
	\begin{thm}
		Let $G$ be a simple mixed graph on $n$ vertices. Then for any $U\subset V_G$, we have $\bm{\nu}_{2n-1}(G)\leq\bm{\nu}_{2k-1}(G-U)+2|U|$, where $k=|V_G|-|U|$.
	\end{thm}
	\begin{proof}
		Let $V_G=\{v_1,v_2,\ldots,v_n\}$. Without loss of generality, let $U=\{v_1,v_2,\ldots,v_k\}$, where $k\in\{1,2,\ldots,n-1\}$. Then $U'=\{v'_1,v'_2,\ldots,v'_k,v''_1,v''_2,\ldots,v''_k\}$ is the set of vertices in $G^A$ corresponding to $U$. Since $U'\subset V_{G^A}$, from \cite[Theorem~7.4.5]{cvetkovic}: 
%		``Let $G$ be a simple graph and let $U\subset V_G$. Then $\nu_{n-1}(G)\leq\nu_{k-1}(G-U)+|U|$, where $k=|V_G|-|U|$'', 
		we have $\nu_{2n-1}(G^A)\leq\nu_{2k-1}(G^A-U')+|U'|$, where $k=|V_G|-|U|$. Since $\bm{\nu}_{2n-1}(G)=\nu_{2n-1}(G^A)$, $\bm{\nu}_{2k-1}(G-U)=\nu_{2k-1}(G^A-U')$, and $|U'|=2|U|$, the result follows.
	\end{proof}
	
	\begin{thm}
		Let $G$ be a simple mixed graph with $E_G\neq \emptyset$ or $\vec{E}_G\neq\emptyset$. Then $\bm{\nu}_1(G)\geq \max\{\Delta_1(G),\Delta_2(G)\}+1.$
	\end{thm}
	\begin{proof}
		Since $E_G\neq \emptyset$ or $\vec{E}_G\neq\emptyset$, the graph $G^A$ has at least one edge. Since $\bm{\nu}_1(G)=\nu_1(G^A)$ and $\max\{\Delta_1(G),\Delta_2(G)\}=\Delta(G^A)$, the proof follows from \cite[Theorem~4.12]{bapatbook}: `` $\nu_1(G)\geq\Delta(G)+1$ for any simple graph $G$ having at least one edge".
	\end{proof}
	We refer to a mixed graph $G$ as \textit{plain} if  it has no multiple edges and no multiple arcs. 
	\begin{lemma}\label{plainsimple}
		A loopless mixed graph $G$ is plain if and only if $G^A$ is simple.
	\end{lemma}
	\begin{proof}
		Let $G$ be a loopless plain mixed graph. Since $G$ has no loops, $G^A$ also has no loops. Moreover, since $G$ has no multiple edges and no multiple arcs, the construction of $G^A$ ensures that $G^A$ has no multiple edges. Therefore, $G^A$ is simple. By retracing the argument, we can prove the converse of this result.
	\end{proof}
%	 The complement $\overline{H}$ of a simple graph $H$, is the simple graph whose vertex set is $V_H$, and two vertices in $\overline{H}$ are
%	adjacent if and only if they are not adjacent in $H$.
	\begin{thm}
	For a loopless plain mixed graph $ G $ with $ n $ vertices, the bound $ \bm{\nu}_1(G) \leq 2n $ holds. Equality is attained if and only if $ G $ is a uniconnected mixed graph whose adjacency graph complement is disconnected.
	\end{thm}
	\begin{proof}
		Under the assumptions,  it follows from Lemma~\ref{plainsimple} that  $G^A$ is a simple graph on $2n$ vertices. Since $\bm{\nu}_1(G)=\nu_1(G^A)$,  the result can be derived using  Lemma~\ref{uniconnected connected} and \cite[Proposition~7.3.3]{cvetkovic}.
%		``If $G$ is a simple graph on $n$ vertices, then $\nu_1(G)\leq n$, with equality for a connected graph $G$ if and only if the complement of $G$ is not connected''. 
	\end{proof}
	
	\begin{thm}
		For any simple mixed graph $G$, $\bm{\nu}_1(G)\leq \max A$, where $A=\{x\mid x=d(u)+d(v)+d^+(u)+d^+(v) \text{ or }x=d(u)+d(v)+d^-(u)+d^-(v) \text{ for } u,v\in V_G\text{ and } u \sim v\}\cup\{x\mid x=d(u)+d(v)+d^+(u)+d^-(v)\text{ for }u,v\in V_G\text{ and }u\rightarrow v\}$.
	\end{thm}
	\begin{proof}
		From \cite[Theorem~7.3.4]{cvetkovic}
%		 ``For any simple graph $G$, $\nu_1(G)\leq \max \{d(u)+d(v)\mid u,v\in V_G \textnormal{ and }u\sim v\}$''
		 , we have $\nu_1(G^A)\leq \max\{d(u)+d(v)\mid u,v\in V_{G^A}\text{ and } u \sim v \}$. 
		Since $\bm{\nu}_1(G)=\nu_1(G^A)$, and the set $\{d(u)+d(v)\mid u,v\in V_{G^A}\text{ and } u \sim v \}$ coincides with the set $A$ defined in the theorem, the result follows.
	\end{proof}
	
	\begin{thm}
		Let $G$ be a simple mixed graph with $V_G=\{v_1,v_2,\ldots,v_n\}$. Let $\mathcal{I}(G)=(a_{ij})$. Then
		{\small\begin{eqnarray*}
				\bm{\nu}_1(G)&\geq & \max\left(\left\{\sqrt{(d(v_i)+d^+(v_i)-d(v_j)-d^+(v_j))^2+4a_{ij}}\mid 1\leq i,j\leq n, i\neq j\right\}\right.\\
				&~&\cup \left\{\sqrt{(d_{v_i}+d^+(v_i)-d(v_{j-n})-d^-(v_{j-n}))^2+4a_{ij}}\mid 1\leq i\leq n, n+1\leq j\leq 2n\right\}\\
				&~&\cup \left\{\sqrt{(d(v_{i-n})+d^-(v_{i-n})-d(v_{j})-d^+(v_{j}))^2+4a_{ij}}\mid n+1\leq i\leq 2n, 1\leq j\leq n\right\}\\
				&~&\cup \left.\left\{\sqrt{(d(v_{i-n})+d^-(v_{i-n})-d(v_{j-n})-d^-(v_{j-n}))^2+4a_{ij}}\mid n+1\leq i,j\leq 2n,i\neq j\right\}\right)
		\end{eqnarray*}}
	\end{thm}
	\begin{proof}
		The vertex degrees of $G^A$ are $d(v_1)+d^+(v_1)$, $d(v_2)+d^+(v_2)$, $\dots$, $d(v_n)+d^+(v_n)$, $d(v_1)+d^-(v_1)$, $d(v_2)+d^-(v_2)$, $\dots$, $d(v_n)+d^-(v_n)$. If we denote the maximum in the right hand side of the inequality of this theorem as $m$, then from \cite[Theorem~7.3.6]{cvetkovic}, 
%		``If $G$ is a simple graph with $V_G=\{v_1,v_2,\ldots,v_n\}$, then $\nu_1(G)\geq \max \{\sqrt{(d(v_i)-d(v_j)^2+4a_{ij})}\mid v_i,v_j\in V_G\textnormal{ and }i\neq j\}$'',
		 we have 
		$\nu_1(G^A)\geq m$. Since $\bm{\nu}_1(G)=\nu_1(G^A)$, the result follows.
	\end{proof}
	
	Let $G$ be a mixed graph on $n$ vertices $v_1,v_2,\ldots,v_n$. We denote for $i=1,2,\ldots,n$,
	\begin{eqnarray*}
		m_i:=avg\left\{\left\{d(v_k)+d^+(v_k)\mid v_i\sim v_k\right\}\cup\left\{d(v_k)+d^-(v_k)\mid v_i\rightarrow v_k\right\}\right\},\\
		M_i:=avg\left\{\left\{d(v_k)+d^-(v_k)\mid v_i\sim v_k\right\}\cup\left\{d(v_k)+d^+(v_k)\mid v_i\leftarrow v_k\right\}\right\}.
	\end{eqnarray*}
where $avg~S$ denotes the average of the numbers in the set S. 
	
	\begin{thm}
		If $G$ is a  simple mixed graph with $V(G)=\{v_1,v_2,\ldots,v_n\}$, then
		\begin{eqnarray*}
			\bm{\nu}_1(G)&\leq& \max\left(\left\{x\mid x=\frac{(d_i+d_i^+)(d_i+d_i^++m_i)+(d_j+d_j^+)(d_j+d_j^++m_j)}{d_i+d_i^++d_j+d_j^+}\right.\right. \\ & & \left.\text{~or~} x=\frac{(d_i+d_i^-)(d_i+d_i^-+M_i)+(d_j+d_j^-)(d_j+d_j^-+M_j)}{d_i+d_i^-+d_j+d_j^-} \text{~for~}v_i\sim v_j\right\}\\
			& & \cup \left.\left\{x\mid x=\frac{(d_i+d_i^+)(d_i+d_i^++m_i)+(d_j+d_j^-)(d_j+d_j^-+M_j)}{d_i+d_i^++d_j+d_j^-} \text{~for~}v_i\rightarrow v_j\right\}\right),
		\end{eqnarray*}	
		where $d_i=d(v_i)$, $d^+_i=d^+(v_i)$, and $d^-_i=d^-(v_i)$ for $i=1,2,\ldots,n$.
	\end{thm}
	\begin{proof}
		The vertex degrees of $G^A$ are $d_{1}+d_{1}^+, d_{2}+d_{2}^+, \ldots, d_{n}+d_{n}^+, d_{1}+d_{1}^-, d_{2}+d_{2}^-, \dots$, $d_{_n}+d_{n}^-$. Observe that $m_i$ and $M_i$ represent the average degrees of the neighbors of the vertices $v'_i$ and $v''_i$ in $G^A$, respectively. Let $s$ denotes the maximum value on the right hand side of the inequality stated in the theorem. Then from \cite[Theorem~7.3.5]{cvetkovic},
%		
%		 ``Let $G$ be a simple graph with $V_G=\{v_1,v_2,\ldots,v_n\}$. For $i=1,2,\ldots,n$, if $m_i=avg\{d(v_k)\mid v_i\sim v_k\}$, then $\nu_1(G)\leq\underset{v_i\sim v_j}{\max}\frac{d(v_i)(d(v_i)+m_i)+d(v_j)(d(v_j)+m_j)}{d(v_i)+d(v_j)}$'', 
		we conclude that
		$\nu_1(G^A)\leq s$. Since $\bm{\nu}_1(G)=\nu_1(G^A)$, the result follows.
	\end{proof}
	
	\section{Integrated signless Laplacian matrix of a mixed graph}\label{S5}
In this section, we define the integrated signless Laplacian matrix of a mixed graph, investigate its properties, and analyze the relationship between its eigenvalues and the structural features of the mixed graph.
	\begin{defn}\normalfont
		The \textit{integrated signless Laplacian matrix} of a mixed graph $G$, denoted by $\mathcal{I}^Q(G)$, is defined as $\mathcal{I}^Q(G)=\mathcal{I}^D(G)+\mathcal{I}(G)$.
	\end{defn}
	
	 Notice that each mixed graph can be determined from its integrated signless Laplacian matrix. The eigenvalues of $\mathcal{I}^Q(G)$ are referred to as the $\mathcal{I}^Q$-\textit{eigenvalues} of $G$. We denote them by $\bm{\xi}_i(G)$ for $i=1,2,\ldots,2n$. Since $\mathcal{I}^Q(G)$ is  real symmetric, its eigenvalues can be arranged, without loss of generality as $\bm{\xi}_1(G)\geq\bm{\xi}_2(G)\geq\cdots\geq\bm{\xi}_{2n}(G)$. The spectrum of $\mathcal{I}^Q(G)$ is called the $\mathcal{I}^Q$-\textit{spectrum} of $G$. Observe that  $\mathcal{I}^Q(G)=Q(G^A)$, which implies that the spectrum of $\mathcal{I}^Q(G)$ is identical to the spectrum of $Q(G^A)$. Moreover, $\mathcal{I}^Q(G)$ is positive semi-definite  because $Q(G^A)$ is positive semi-definite. 
	
Observe that using~\eqref{adjmatrix block}, $\mathcal{I}^Q(G)$ can be viewed as a $2\times 2$ block matrix
	\begin{equation}\label{signlessblock}
		\mathcal{I}^Q(G)=\begin{bmatrix}
			Q(G_u)+D_1 & \vec{A}(G_d)\\
			\vec{A}(G_d)^T & Q(G_u)+D_2
		\end{bmatrix},
	\end{equation}
	where $D_1=diag(d^+(v_1),d^+(v_2),\ldots,d^+(v_n))$ and $D_2=diag(d^-(v_1),d^-(v_2),\ldots,d^-(v_n))$. 
%	\begin{example}\normalfont
%		The signless Laplacian matrix of the mixed graph $G$ shown in Figure~\ref{fig-mixed-exmpl}(i) is\\ $\mathcal{I}^Q(G)=\begin{bmatrix}
%			9&1&1&1&1&1&0&0\\
%			1&12&0&2&2&2&0&1\\
%			1&0&12&0&1&1&0&1\\
%			1&2&0&4&0&0&1&0\\
%			1&2&1&0&11&1&1&1\\
%			1&2&1&0&1&11&0&2\\
%			0&0&0&1&1&0&10&0\\
%			0&1&1&0&1&2&0&5
%		\end{bmatrix}$.
%	\end{example}
	
	\subsection{Results}
	\begin{observation}
		Let $G$ be a mixed graph on $n$ vertices. Then we have the following.
		\begin{enumerate}[(1)]
			\item $G$ is $r$-regular if and only if $2r$ is an eigenvalue of $\mathcal{I}^Q(G)$ with corresponding eigenvector $\boldsymbol{1}_{2n}$.
			%	\item If $\begin{bmatrix}
				%			\boldsymbol{1}_n\\
				%			-\boldsymbol{1}_n
				%		\end{bmatrix}$ is an eigenvector of $\mathcal{I}^Q(G)$ with corresponding eigenvalue $\xi$, then $\displaystyle d(u)=\frac{\xi}{2}$.
			\item If $G$ is a graph, then $\mathcal{I}^Q(G)=I_2\otimes Q(G)$. The eigenvalues of $\mathcal{I}^Q(G)$ are identical to those of $Q(G)$ but with twice their multiplicities.
		\end{enumerate}
	\end{observation}
	\begin{thm}\label{thm-trace Q}
		Let $G$ be a mixed graph on $n$ vertices. Then 
		$$\underset{i=1}{\overset{2n}{\sum}}\bm{\xi}_i(G)=4e(G)+8l(G)+2a(G).$$
	\end{thm}
	\begin{proof}
		Let $V_G=\{v_1,v_2,\ldots,v_n\}$. Then $\underset{i=1}{\overset{2n}{\sum}}\bm{\xi}_i(G)=tr(\mathcal{I}^Q(G))=tr(\mathcal{I}^D(G)+\mathcal{I}(G))=tr(\mathcal{I}^D(G))+tr(\mathcal{I}(G))=\underset{i=1}{\overset{n}{\sum}}(2d(v_i)+d^+(v_i)+d^-(v_i))+\underset{i=1}{\overset{n}{\sum}}4l(v_i)$. 
		In this equations, by substituting the values  $\underset{i=1}{\overset{n}{\sum}}d(v_i)=2e(G)+2l(G)$, and
		$\underset{i=1}{\overset{n}{\sum}}d^+(v_i)=a(G)=\underset{i=1}{\overset{n}{\sum}}d^-(v_i)$ (c.f \cite[Proposition~4.1]{Kalaimatrices1}), the result follows.
	\end{proof}
	
	\begin{thm}\label{Th3.2}
		Let $G$ be a simple mixed graph having AB property. Then the characteristic polynomial of $\mathcal{I}^Q(G)$ coincides with the characteristic polynomial of $\mathcal{I}^L(G)$.
	\end{thm}
	\begin{proof}
		 Since $\mathcal{I}^Q(G)=Q(G^A)$ and $\mathcal{I}^L(G)=L(G^A)$, the result follows from   \cite[Lemma~4.1]{Kalaimatrices1}: ``Let $G$ be a mixed graph. $G$ has AB property if and only if $G^A$ is bipartite'' and \cite[Proposition~7.8.4]{cvetkovic}.
%		  ``For any simple bipartite graph $G$, the characteristic polynomial of $Q(G)$ coincides with the characteristic polynomial of $L(G)$''. 
	\end{proof}
	As an immediate consequence of the previous theorem, we get the following.
	\begin{cor}
		For any simple bipartite mixed graph $G$, the characteristic polynomial of $\mathcal{I}^Q(G)$ coincides with the characteristic polynomial of $\mathcal{I}^L(G)$.
	\end{cor}
	%\begin{proof}
	%	For any bipartite mixed graph $G$, $G^A$ is a bipartite graph. So the result follows from Theorem~\ref{Th3.2}.
	%\end{proof}
	
	\begin{thm}\label{Th3.4}
	If $ G $ is a simple mixed graph that is $ r $-regular, then $ r $ is given by $ r = \frac{1}{2} \bm{\xi}_1(G) $. Furthermore, the multiplicity of $\bm{\xi}_1(G) $ equals the number of mixed components in $ G $.
	\end{thm}
	\begin{proof}
		From \cite[Lemma~6.1]{Kalaimatrices1}: ``Let $G$ be a mixed graph. Then $G$ is $r$-regular if and only if $G^A$ is $r$-regular'', we know that $G^A$ is an $r$-regular simple graph. According to  \cite[Theorem~7.8.6]{cvetkovic}, 
%		 ``Let $G$ be a simple graph. If $G$ is $r$-regular, then $r=\frac{1}{2}\xi_1(G)$, and the number of components of $G$ equals the multiplicity of $\xi_1(G)$'', 
		 we conclude that $r=\frac{1}{2}\xi_1(G^A)=\frac{1}{2}\bm{\xi}_1(G)$, and  the multiplicity of $\xi_1(G^A)$ equals the number of components of $G^A$. Thus,  the result follows from \cite[Corollary~6.1]{Kalaimatrices1}: ``The number of mixed components of a mixed graph $G$ is equal to the number of components of $G^A$''.
	\end{proof}
	\begin{cor}
		If $G$ is an $r$-regular simple uniconnected mixed graph, then $\bm{\xi}_2(G)<\bm{\xi}_1(G)=2r$.
	\end{cor}
	\begin{proof}
		Since $G$ is uniconnected, it contains exactly one mixed component.
		Therefore,  from Theorem~\ref{Th3.4}, we know that $\bm{\xi}_1(G)=2r$ with multiplicity  $1$. This implies that $\bm{\xi}_1(G)\neq\bm{\xi}_2(G)$, so we can conclude that $\bm{\xi}_2(G)<\bm{\xi}_1(G)$.
	\end{proof}
	
	\begin{thm}\label{thm-regular signless spectrum}
		The $\mathcal{I}^Q$-spectrum of an $r$-regular mixed graph $G$ on $n$ vertices is given by $2r=r+\bm{\lambda}_1(G)\geq r+\bm{\lambda}_2(G)\geq\cdots\geq r+\bm{\lambda}_{2n}(G)$.
	\end{thm}
	\begin{proof}
		Since $G$ is $r$-regular, we have $\mathcal{I}^Q(G)=rI_{2n}+\mathcal{I}(G)$ and $\bm{\lambda}_1(G)=r$. Hence the result follows.
	\end{proof}
	
	\begin{thm}
		Let $G$ be a mixed graph on $n$ vertices. Then $G$ is $(r,s)$-regular if and only if $2(r+s)$ and $2r$ are  $\mathcal{I}^Q$-eigenvalues of $G$ with corresponding eigenvectors $\boldsymbol{1}_{2n}$ and $\begin{bmatrix}
			\boldsymbol{1}_n^T&
			-\boldsymbol{1}_n^T
		\end{bmatrix}^T$, respectively.
	\end{thm}
	\begin{proof}
		Suppose that $G$ is $(r,s)$-regular.  Then by
		\cite[Theorem~5.1]{Kalaimatrices1}, $r+s$ and $r-s$ are  eigenvalues of $\mathcal{I}(G)$ with corresponding eigenvectors $\boldsymbol{1}_{2n}$ and $\begin{bmatrix}
		\boldsymbol{1}_n^T&
		-\boldsymbol{1}_n^T
		\end{bmatrix}^T$, respectively. Observe that $\mathcal{I}^Q(G)=(r+s)I_{2n}+\mathcal{I}(G)$. Therefore,  $2(r+s)$ and $2r$ are eigenvalues of $\mathcal{I}^Q(G)$ with corresponding eigenvectors $\boldsymbol{1}_{2n}$ and $\begin{bmatrix}
			\boldsymbol{1}_n^T&
			-\boldsymbol{1}_n^T
		\end{bmatrix}^T$, respectively. 
		
		Conversely, we assume that $2(r+s)$ and $2r$ are  eigenvalues of $\mathcal{I}^Q(G)$ with corresponding eigenvectors $\boldsymbol{1}_{2n}$ and $\begin{bmatrix}
			\boldsymbol{1}_{n}^T&
			-\boldsymbol{1}_n^T
		\end{bmatrix}^T$, respectively. Then $d(u)+d^+(u)=r+s$, $d(u)+d^-(u)=r+s$ and $d(u)=r$ for all $u\in V_G$. These implies that $d(u)=r$ and $d^+(u)=d^-(u)=s$ for all $u\in V_G$. So $G$ is $(r,s)$-regular. 
	\end{proof}
	
	%\begin{thm}
	%	For a graph $G$, the spectrum of $\mathcal{I}^Q(G)$ is same as the spectrum of $Q(G)$ with twice the multiplicities.
	%\end{thm}
	%\begin{proof}
	%	For a  graph $G$, we have $\mathcal{I}^Q(G)=I_2\otimes Q(G)$. So, the result follows.
	%\end{proof}
	
	%\begin{thm}
	%	Let $G$ be a mixed graph with $d^+(v)+d^-(v)+d(v)=1$ for each vertex $v$ of $G$. Then the eigenvalues of $\mathcal{I}^Q(G)$ are $2$ with multiplicity $2e(G)+a(G)$ and $0$ with multiplicity $2e(G)+3a(G)$.
	%\end{thm}
	%\begin{proof} 
	%	By rearranging the rows and columns of $\mathcal{I}^Q(G)$, we obtain the matrix
	%	$$\begin{bmatrix}
		%		I_{2e(G)+a(G)} & \mathbf{0}\\
		%		\mathbf{0} & \mathbf{0}_{a(G)}
		%	\end{bmatrix}\otimes J_2.$$ From Lemma~\ref{lem-cronecker}, the eigenvalues of the above matrix are $2$ and $0$ with multiplicities $2e(G)+a(G)$ and $2e(G)+3a(G)$, respectively.
	%\end{proof}
	
	\begin{thm}\label{signlesslaplaciancomplete}
		\begin{enumerate}[(i)]
			\item The $\mathcal{I}^Q$-spectrum of $K^M_{k(m)}$ is $(4mk-4m)^{(1)}$, $(2mk-4m)^{(k-1)}$, $(2mk-2m)^{(2mk-k)}$.
			\item The $\mathcal{I}^Q$-spectrum of $K^D_{k(m)}$ is $(2mk-2m)^{(1)}$, $0^{(1)}$, $(mk)^{(k-1)}$, $(mk-2m)^{(k-1)}$, $(mk-m)^{(2km-2k)}$.
		\end{enumerate}
	\end{thm}
	\begin{proof}
		 From Theorem~\ref{thm-regular signless spectrum}, the result can be proved similar to the proof of Theorem~\ref{laplaciancomplete}.
	\end{proof}
	
	\begin{cor}
		\begin{itemize}
			\item[(i)] The $\mathcal{I}^Q$-spectrum of $K^M_{n}$ is $(4n-4)^{(1)}$, $(2n-4)^{(n-1)}$, $(2n-2)^{(n)}$.
			\item[(ii)] The $\mathcal{I}^Q$-spectrum of $K^D_n$ is $(2n-2)^{(1)}$, $0^{(1)}$, $n^{(n-1)}$, $(n-2)^{(n-1)}$.
		\end{itemize}
	\end{cor}
	\begin{proof}
		By taking $k=n$ and $m=1$ in Theorem~\ref{signlesslaplaciancomplete}, the result follows.
	\end{proof}
	\begin{thm}
		\begin{enumerate}[(i)]
			\item If $G$ is an oriented graph of $P_n$ $(n\geq 2)$ with all its arcs have the same direction, then the signless Laplacian spectrum of $G$ is $2^{(n-1)}$, $0^{(n+1)}$.
			\item If $G$ is an oriented graph of $C_n$ $(n\geq 3)$ with all its arcs have the same direction, then the signless Laplacian spectrum of $G$ is $2^{(n)}$, $0^{(n)}$.
			\item If $G$ is an oriented graph of $C_n$ $(n\geq 4)$ with all its arcs have the alternating direction, then the signless Laplacian spectrum of $G$ is $0^{(n)}$, and $(2+2\cos\frac{2\pi k}{n})^{(1)}$ for $k=1,2,\ldots,n$.
		\end{enumerate}
	\end{thm}
	\begin{proof}
		Proof is similar to the proof of  Theorem~\ref{orientedLaplacian}.
	\end{proof}
	\begin{thm}Let $G$ be a uniconnected mixed graph on $n$ vertices. If $G$ has an alternating cycle of length $2n$ such that it contains all the vertices, all the arcs and two times all the edges of $G$, then the $\mathcal{I}^Q$-spectrum of $G$ is $2-2\cos\frac{\pi k}{n}$ for $k=1,2,\ldots,2n$.
	\end{thm}
	\begin{proof}
		From Theorem~\ref{thm-regular signless spectrum}, the result can be proved similar to the proof of Theorem~\ref{Laplacian spectrum cycle}.
	\end{proof}
	\begin{thm}\label{thm3.5}
		Let $G$ be a uniconnected simple mixed graph on $n$ vertices. Then $G$ has AB property if and only if $\bm{\xi}_{2n}(G)=0$, and is a simple eigenvalue.
	\end{thm}
	\begin{proof}
		By Lemma~\ref{uniconnected connected}, $ G^A $ is a non-trivial, simple, connected graph on $ 2n $ vertices. Since, $\bm{\xi}_{2n}(G)=\xi_{2n}(G^A)$ with the same multiplicity,  the proof follows from \cite[Lemma~4.1]{Kalaimatrices1}: ``Let $G$ be a mixed graph. $G$ has AB property if and only if $G^A$ is bipartite''  and \cite[Theorem~7.8.1]{cvetkovic}.
%		
%		 ``Let $G$ be a simple non-trivial connected graph on $n$ vertices. Then $G$ is bipartite if and only if $\xi_{n}(G)=0$. In this case, $0$ is a simple eigenvalue of $Q(G)$''.
	\end{proof}
	\begin{lemma}\label{ABcomponentbipartite}
		Let $G$ be a mixed graph. Then a non-trivial mixed component of $G$ has the AB property if and only if $G^A$ has a bipartite component.
	\end{lemma}
	\begin{proof}
		Let $H$ be a non-trivial mixed component of $G$ that possesses the AB property, and let $H'$ denote the associated component of $H$ in $G^A$. Then $H'$ is non-trivial. By \cite[Lemma~4.1]{Kalaimatrices1}: ``Let $G$ be a mixed graph. $G$ has AB property if and only if $G^A$ is bipartite'', $H^A$ is bipartite. As $H'$ is a subgraph of $H^A$, it is also bipartite. Thus, $H'$ is a bipartite component of $G^A$.
		
		Conversely, suppose $G^A$ has a bipartite component $H'$. Let $H$ be the mixed component in $G$ corresponding to $H'$. Then $H$ is non-trivial. The bipartiteness of $H'$ implies it has no odd cycles. Consequently,  $H$ has no odd cycles and no odd alternating cycles with an even number of arcs. This ensures that $H$ has AB property. Therefore, $G$ has a non-trivial mixed component having the AB property. 
	\end{proof}
	\begin{thm}\label{thm3.6}
		For any simple mixed graph $G$, the multiplicity of $0$ as an $\mathcal{I}^Q$-eigenvalue of $G$ is equal to the number of mixed components of $G$ that possess AB property.
	\end{thm}
	\begin{proof}
		From Lemma~\ref{ABcomponentbipartite}, the number of trivial mixed components of $G$ having the AB property is the same as the number of bipartite components of $G^A$. Additionally, the number of trivial mixed components of $G$ matches the number of trivial components of $G^A$. It is evident that trivial mixed components of $G$ exhibit the AB property. Since  $\mathcal{I}^Q(G)=Q(G^A)$,  the result follows directly from~\cite[Corollary~7.8.2]{cvetkovic}.
%		
%		 ``For any simple graph, the multiplicity of $0$ as an eigenvalue of $Q(G)$ is equal to the number of components of $G$ that are bipartite or trivial''.
	\end{proof}
	
	\begin{cor}
	For a simple bipartite mixed graph $ G $ with $ n $ vertices, the eigenvalue $ \bm{\xi}_{2n}(G) $ is zero, and its multiplicity corresponds to the number of mixed components in $ G $.
	\end{cor}
	\begin{proof}
		Since $G$ is a simple bipartite mixed graph on $n$ vertices, 
		each mixed component of $G$ has AB property.
		So, the result follows from Theorem~\ref{thm3.6}.
	\end{proof}
	%\begin{cor}
	%	For each digraph $G$ with $n$ vertices, $\xi_{2n}(G)=0$ with multiplicity is equal to the number of ASMs of $G$.
	%\end{cor}
	%\begin{proof}
	%	Since $G$ is a digraph on $n$ vertices, $G^A$ is a bipartite graph on $2n$ vertices. So, each component of $G^A$ is bipartite or trivial. Since the number of ASMs of $G$ is equal to the number of components of $G^A$, the result follows from Theorem~\ref{thm3.6}.
	%\end{proof}
	
	\begin{thm}
		Let $G$ be a simple mixed graph on $n$ vertices and let $a\in \vec{E}_G$. If $H=G-a$, then $0\leq\bm{\xi}_{2n}(G')\leq\bm{\xi}_{2n}(G)\leq\cdots\leq\bm{\xi}_{2}(H)\leq\bm{\xi}_{2}(G)\leq\bm{\xi}_{1}(H)\leq\bm{\xi}_{1}(G)$.
	\end{thm}
	\begin{proof}
		Since $H=G-a$, we have ${H}^A=G^A-e$, where $e$ is the edge in $G^A$ corresponding to the arc $a$ in $G$. Since $\bm{\xi}_{i}(G)=\xi_{i}(G^A)$ and $\bm{\xi}_{i}(H)=\xi_{i}(H^A)$ for $i=1,2,\ldots,2n$, the proof follows from \cite[Theorem~7.8.13]{cvetkovic}.
%		
%		``Let $G$ be a simple graph on $n$ vertices and let $e\in E_G$. If $G'=G-e$, then $0\leq\xi_{n}(G')\leq\xi_{n}(G)\leq\cdots\leq\xi_{2}(G')\leq\xi_{2}(G)\leq\xi_{1}(G')\leq\xi_{1}(G)$''. 
	\end{proof}
	
	\subsubsection{Bounds}
	
	\begin{thm}
		For a mixed graph $G$ on $n$ vertices,
		$$\bm{\xi}_{2n}(G)\leq\frac{2e(G)+4l(G)+a(G)}{n}\leq\bm{\xi}_{1}(G).$$ 
	\end{thm}
	\begin{proof}
		By Theorem~\ref{thm-trace Q}, $\underset{i=1}{\overset{2n}{\sum}}\bm{\xi}_i(G)=4e(G)+8l(G)+2a(G$. Since $\bm{\xi}_{2n}(G)\leq\bm{\xi}_i(G)\leq\bm{\xi}_1(G)$ for $i=1,2,\ldots,2n$, it follows that $2n\bm{\xi}_{2n}(G)\leq\underset{i=1}{\overset{2n}{\sum}}\bm{\xi}_i(G)\leq2n\bm{\xi}_1(G)$. Therefore, the inequalities in this theorem are satisfied.
	\end{proof}
	
	\begin{thm}
		For any simple mixed graph $G$, we have $$\min\{\delta_1(G),\delta_2(G)\}\leq\frac{1}{2}\bm{\xi}_1(G)\leq\max\{\Delta_1(G),\Delta_2(G)\}.$$ 
		For a uniconnected mixed graph $ G $, equality is achieved in either case if and only if $ G $ is regular.
	\end{thm}
	\begin{proof} 
		Since $\bm{\xi}_1(G)=\xi_1(G^A)$, $\min\{\delta_1(G),\delta_2(G)\}=\delta(G^A)$ and $\max\{\Delta_1(G),\Delta_2(G)\}=\Delta(G^A)$, the inequalities in this theorem hold from \cite[Proposition~7.8.14]{cvetkovic}. 
		
%		``For any simple graph $G$, we have $2\delta(G)\leq\xi_1(G)\leq 2\Delta(G)$ and when $G$ is connected, the equality holds in either place if and only if $G$ is regular''.
		 From \cite[Lemma~6.1]{Kalaimatrices1}: ``Let $G$ be a mixed graph. Then $G$ is $r$-regular if and only if $G^A$ is $r$-regular'' and Lemma~\ref{uniconnected connected}, the remaining assertion of this theorem holds.
	\end{proof}
	
	Let $G$ be a mixed graph with $V_G=\{v_1,v_2,\ldots,v_n\}$. Consider the multiset of values $$\{x\mid x=d(v_i)+d^+(v_i)+2l(v_i) ~\text{or}~ x=d(v_i)+d^-(v_i)+2l(v_i) ~\text{for}~i=1,2,\ldots,n\}.$$ Arrange these values in decreasing order to form the sequence $d_1\geq d_2\geq \cdots\geq d_{2n}$. Observe that $d_i$s are the diagonal entries of the matrix $\mathcal{I}^Q(G)$. Then the following result holds:
	\begin{thm}
	For a mixed graph $ G $ with $ n $ vertices, the inequality $ \sum_{i=1}^{k} \bm{\xi}_i(G) \geq \sum_{i=1}^{k} d_i $ holds for $ k = 1,2, \dots, 2n $, with equality attained when $ k = 2n $.
	\end{thm}
	\begin{proof}
%		 From \cite[Theorem~1.3.2]{cvetkovic}: 
%		 ``Let $M$ be a positive semi-definite matrix with eigenvalues $\lambda_1\geq\lambda_2\geq\cdots\geq\lambda_n$. Then $\underset{i=1}{\overset{r}{\sum}}\lambda_i$ is bounded below by the sum of the $r$ largest diagonal entries of $M$''. 
		 Applying   \cite[Theorem~1.3.2]{cvetkovic} to the positive semi-definite matrix $\mathcal{I}^Q(G)$, we deduce that the sum of the $k$ largest diagonal entries of $\mathcal{I}^Q(G)$ is a lower bound of $\underset{i=1}{\overset{k}{\sum}}\bm{\xi}_i(G)$. Moreover, $\underset{i=1}{\overset{2n}{\sum}}\bm{\xi}_i(G)=tr(\mathcal{I}^Q(G))=\underset{i=1}{\overset{2n}{\sum}}d_i$. This completes the proof.
	\end{proof}
	
	\begin{thm}\label{thm-LaplacianQ}
		Let $G$ be a simple mixed graph on $n$ vertices. Then $$\bm{\xi}_1(G)+\bm{\xi}_{2n}(G)\leq 2\xi_1(G_u)+\underset{u\in V_G}{\max}~d^+(u)+\underset{u\in V_G}{\max}~d^-(u).$$
	\end{thm}
	\begin{proof}
	 We consider $\mathcal{I}^Q(G)$ as mentioned in~\eqref{signlessblock}.	
		By taking $M=\mathcal{I}^Q(G)$, $P=Q(G_u)+D_1$, $Q=\vec{A}(G_d)$ and $R=Q(G_u)+D_2$ in \cite[Proposition~1.3.16]{cvetkovic}, we get \begin{equation}\label{thm-LaplacianQeq1}
		\bm{\xi}_1(G)+\bm{\xi}_{2n}(G)\leq \lambda_1(Q(G_u)+D_1)+\lambda_1(Q(G_u)+D_2).
		\end{equation}
		%, where $\lambda_1(Q(G_u)+D_1)$ and $\lambda_1(Q(G_u)+D_2)$ denote the largest eigenvalues of $Q(G_u)+D_1$ and $Q(G_u)+D_2$ respectively. 
		From \cite[Lemma~3.19]{bapatbook}, it follows that \begin{equation}\label{thm-LaplacianQeq2}
			\displaystyle\lambda_1(Q(G_u)+D_1)\leq \xi_1(Q(G_u))+\underset{v_i\in V_G}{\max}~d^+(v_i)
		\end{equation}
		 and 
		 \begin{equation}\label{thm-LaplacianQeq3}
		 	\displaystyle\lambda_1(Q(G_u)+D_2)\leq \xi_1(Q(G_u))+\underset{v_i\in V_G}{\max}~d^-(v_i).\end{equation}
		 The result follows by substituting \eqref{thm-LaplacianQeq2} and \eqref{thm-LaplacianQeq3} into \eqref{thm-LaplacianQeq1}.
	\end{proof}
	
	\begin{thm}
		Let $G$ be a mixed graph on $n$ vertices. Then $\bm{\xi}_{2n}(G)\leq\displaystyle\frac{4}{n}(e(G)+l(G))\leq\bm{\xi}_2(G)$ and 
		$\bm{\xi}_{2n-1}(G)\leq\displaystyle\frac{2}{n}(2e(G)+2l(G)+a(G))\leq\bm{\xi}_1(G)$.
	\end{thm}
	\begin{proof}
		Consider $\mathcal{I}^Q(G)$ as mentioned in \eqref{signlessblock}.
		The sums of the entries of $Q(G_u)+D_1$, $Q(G_u)+D_2$, and $\vec{A}(G_d)$ are $4e(G)+4l(G)+a(G)$, $4e(G)+4l(G)+a(G)$, and $a(G)$, respectively. From \cite[Corollary~1.3.13]{cvetkovic}, the eigenvalues $\frac{2}{n}(2e(G)+2l(G)+a(G))$ and $\frac{4}{n}(e(G)+l(G))$ of the matrix $\frac{4}{n}(e(G)+l(G))I_2+\frac{1}{n}a(G)J_2$ interlace the eigenvalues of $\mathcal{I}^Q(G)$. This concludes the proof.
	\end{proof}
	
	\begin{thm}\label{thm-factorization ineqq}
		If a simple mixed graph $G$ on $n$ vertices has a  factorization $G=G_1\oplus G_2$, then the following hold:
		\begin{enumerate}[(i)]
			\item $\max\{\bm{\xi}_{1}(G_1),\bm{\xi}_{1}(G_2)\}\leq\bm{\xi}_{1}(G)\leq\bm{\xi}_{1}(G_1)+\bm{\xi}_{1}(G_2)$,
			\item $\bm{\xi}_{2n}(G_1)+\bm{\xi}_{2n}(G_2)\leq \bm{\xi}_{2n}(G)$.
		\end{enumerate}
	\end{thm}
	\begin{proof}
		Since $G=G_1\oplus G_2$, we have $\mathcal{I}^Q(G)=\mathcal{I}^Q(G_1)+\mathcal{I}^Q(G_2)$. Let $X$ be a column vector in $\mathbb{R}^{2n}$ with the usual Euclidean norm. Then by extremal representation of eigenvalues:
		\begin{eqnarray*}
			\bm{\xi}_1(G)&=&\underset{||X||=1}{\max}X^T\mathcal{I}^Q(G)X\\
			&=&\underset{||X||=1}{\max}(X^T\mathcal{I}^Q(G_1)X+X^T\mathcal{I}^Q(G_2)X)\\
			&\leq&\underset{||X||=1}{\max}X^T\mathcal{I}^Q(G_1)X+\underset{||X||=1}{\max}X^T\mathcal{I}^Q(G_2)X\\
			&=&\bm{\xi}_1(G_1)+\bm{\xi}_1(G_2).
		\end{eqnarray*}
		Moreover, since
	\begin{equation*}
		\bm{\xi}_1(G)=\underset{||X||=1}{\max}(X^T\mathcal{I}^Q(G_1)X+X^T\mathcal{I}^Q(G_2)X), 		
	\end{equation*}		
		and both $\mathcal{I}^Q(G_1)$ and $\mathcal{I}^Q(G_2)$ are positive semi-definite,
		we also have 
		\begin{equation*}
			\bm{\xi}_1(G)\geq\underset{||X||=1}{\max}X^T\mathcal{I}^Q(G_1)X=\bm{\xi}_1(G_1)
		\end{equation*}		
			 and 	
		\begin{equation*}
			 		\bm{\xi}_1(G)\geq\underset{||X||=1}{\max}X^T\mathcal{I}^Q(G_2)X=\bm{\xi}_1(G_2).
		\end{equation*}	
	Thus, part~(i) is established.
		
		For part~(ii), using the extremal representation of the smallest eigenvalue, we have:
		\begin{eqnarray*}
			\bm{\xi}_{2n}(G)&=&\underset{||X||=1}{\min}X^T\mathcal{I}^Q(G)X\\
			&=&\underset{||X||=1}{\min}(X^T\mathcal{I}^Q(G_1)X+X^T\mathcal{I}^Q(G_2)X)\\
			&\geq&\underset{||X||=1}{\min}X^T\mathcal{I}^Q(G_1)X+\underset{||X||=1}{\min}X^T\mathcal{I}^Q(G_2)X\\
			&=&\bm{\xi}_{2n}(G_1)+\bm{\xi}_{2n}(G_2).
		\end{eqnarray*}
		This completes the proof of part~(ii).
	\end{proof}
	\begin{cor}\label{cor-spanning ineqq}
		Let $G$ be a simple mixed graph on $n$ vertices. If $H$ is a spanning submixed graph of $G$, then $\bm{\xi}_{1}(H)\leq\bm{\xi}_{1}(G)$ and $\bm{\xi}_{2n}(H)\leq\bm{\xi}_{2n}(G)$.
	\end{cor}
	\begin{proof}
	We can express  $G=H\oplus H'$, where $H'$ is the spanning submixed graph of $G$ with $E_{H'}=E_G\setminus E_H$ and $\vec{E}_{H'}=\vec{E}_G\setminus \vec{E}_H$. The result then follows directly from Theorem~\ref{thm-factorization ineqq}.
	\end{proof}
	\begin{cor}
		If $G_1$ and $G_2$ are two simple mixed graphs having the same vertex set, then 
		$$\max\{\bm{\xi}_{1}(G_1),\bm{\xi}_{1}(G_2)\}\leq\bm{\xi}_{1}(G_1\cup G_2)\leq\bm{\xi}_{1}(G_1)+\bm{\xi}_{1}(G_2).$$
	\end{cor}
	\begin{proof}
		Since $G_1$ and $G_2$ are spanning submixed graphs of $G_1\cup G_2$, it follows from Corollary~\ref{cor-spanning ineqq} that $\bm{\xi}_{1}(G_1)\leq\bm{\xi}_{1}(G_1\cup G_2)$ and $\bm{\xi}_{2}(G_1)\leq\bm{\xi}_{1}(G_1\cup G_2)$. Thus, the first inequality in the result holds. We can write $G_1\cup G_2=G_1\oplus G'_1$, where $G'_1$ is the spanning submixed graph of $G_2$ with $E_{G'_1}=E_{G_2}\setminus E_{G_1}$ and $\vec{E}_{G'_1}=\vec{E}_{G_2}\setminus \vec{E}_{G_1}$. From Theorem~\ref{thm-factorization ineqq}, we have $\bm{\xi}_{1}(G_1\cup G_2)\leq\bm{\xi}_{1}(G_1)+\bm{\xi}_{1}(G'_1)\leq\bm{\xi}_{1}(G_1)+\bm{\xi}_{1}(G_2)$. Hence, the second inequality of the result also holds.
	\end{proof}

	\begin{thm}
		Let $G$ be a simple mixed graph. Then we have the following.
		\begin{itemize}
			\item[(i)] $\bm{\xi}_1(G)=0$ if and only if $G$ has no edges and arcs;
			\item[(ii)] $\bm{\xi}_1(G)<4$ if and only if each mixed component of $G$ has AP property;
			\item[(iii)] for a uniconnected mixed graph $G$, $\bm{\xi}_1(G)=4$ if and only if $G$ has an alternating cycle with even number of arcs such that it contains all the vertices, all the arcs and two times all the edges of $G$.
		\end{itemize}
	\end{thm}
	\begin{proof}
		From \cite[parts~(i), (ii), (iii) of Proposition~7.8.16]{cvetkovic}, the following results hold:
		\begin{itemize}
	\item $G^A$ has no edges  if and only if $\xi_1(G^A)=0$.
		
		\item All the components of $G^A$ are paths	   if and only if $\xi_1(G^A)<4$. 
		
			\item When $G^A$ is connected, $G^A$ is a cycle or $K_{1,3}$ if and only if $\xi_1(G^A)=4$. 
		\end{itemize}
			
		Since  $\bm{\xi}_1(G)=\xi_1(G^A)$ and $G$ has no edges and arcs if and only if $G^A$ has no edges, part~(i) follows.
		
		From \cite[Lemma~7.2]{Kalaimatrices1}: ``Let $G$ be a mixed graph. A mixed component of $G$ has AP property if and only if the corresponding component of $G^A$ is a path'', part~(ii) follows.

	Now	suppose $G$ is uniconnected. By Lemma~\ref{uniconnected connected}, $G^A$ is connected. This implies that $G$ has an alternating cycle with an even number of arcs that includes all vertices, all arcs, and two instances of every edge of $G$ if and only if $G^A$ is a cycle. Additionally, for any mixed graph $G$, it can be verified that  $G^A\neq K_{1,3}$. Thus, part~(iii) follows.
	\end{proof}
	\begin{defn}\normalfont
		A mixed graph $G$ is said to be \textit{directed loop complete} if for any two distinct vertices $u,v\in V_G$, $u\overset{1}{\sim}v$, $u\overset{1}{\rightarrow}v$ and $v\overset{1}{\rightarrow}u$, and for each vertex $u\in V_G$, $u\overset{1}{\rightarrow}u$.
	\end{defn}
	\begin{lemma}\label{complete}
		A mixed graph $G$ is directed loop complete if and only if $G^A$ is complete.
	\end{lemma}
	\begin{proof}
		Let $G$ be a directed loop complete mixed graph with $V_G=\{v_1,v_2,\ldots,v_n\}$. It is clear that $E_G$ and $\vec{E}_G$ are sets and so, $G$ is plain. As $G$ is loopless, from Lemma~\ref{plainsimple}, $G^A$ is simple. For $i=1,2,\ldots,n$, we have $d(v_i)=n-1$, $d^+(v_i)=n$ and $d^-(v_i)=n$. This implies that for $i=1,2,\ldots,n$, $d(v'_i)=d(v_i)+d^+(v_i)=2n-1$ and  $d(v''_i)=d(v_i)+d^+(v_i)=2n-1$ in $G^A$. That is each vertex of $G^A$ has degree $2n-1$. Thus $G^A$ is complete.
		By retracing the above arguments, the converse of this theorem can be proved.
	\end{proof}
	\begin{thm}
		For a uniconnected loopless plain mixed graph $G$ on $n$ vertices, $2+2\cos{\frac{\pi}{2n}}\leq\bm{\xi}_1(G)\leq 4n-2$. Equality in the lower bound occurs if $G$ satisfies the AP property, while the upper bound is reached when $G$ is a directed loop-complete graph.
	\end{thm}
	\begin{proof}
		Since $\bm{\xi}_1(G)=\xi_1(G^A)$, the inequalities in this theorem hold from Lemma~\ref{uniconnected connected} and \cite[Proposition~7.8.17]{cvetkovic}.
%		 ``Let $G$ be a simple connected graph on $n$ vertices. Then $2+2\cos{\frac{\pi}{n}}\leq\xi_1(G)\leq 2n-2$. The lower bound is attained for $P_n$, and the upper bound for $K_n$''. 
		From \cite[Lemma~7.2]{Kalaimatrices1}: ``Let $G$ be a mixed graph. A mixed component of $G$ has AP property if and only if the corresponding component of $G^A$ is a path'' and Lemma~\ref{complete}, the remaining assertion of this theorem hold. 
	\end{proof}
	
	\section{Normalized integrated Laplacian matrix of a mixed graph}\label{S6}
In this section, we introduce the normalized integrated Laplacian matrix of a mixed graph, study its properties, and analyze the relationship between its eigenvalues and the structural characteristics of the mixed graph.
	\begin{defn}\normalfont
		Let $G$ be a mixed graph with $V_G=\{v_1,v_2,\ldots,v_n\}$. For $i=1,2,\ldots,n$, let $v'_i$ and $v''_i$ be two copies of $v_i$. We define the \textit{normalized integrated Laplacian matrix} of $G$, denoted by $\mathcal{I}^\mathcal{L}(G)$, as the square matrix of order $2n$, with rows and columns are indexed by the elements of the set $\{v'_1,v'_2,\ldots,v'_n,v''_1,v''_2,\ldots,v''_n\}$.
		For $i,j=1,2,\ldots,n$,
		
		the $(v'_i,v'_j)$-th entry of $\mathcal{I}^\mathcal{L}(G)=\begin{cases}
			1-\frac{2l(v_i)}{d_i+d^+_i}, & \text{if~} i=j \text{~and~} d_i+d^{+}_i\neq 0;\\
			\frac{-\beta}{\sqrt{(d^{+}_{i}+d_{i})(d^{+}_{j}+d_{j})}}, & \text{if}~i\neq j\textnormal{ and}~ v_{i}\overset{\beta}{\sim} v_{j};\\
			0, & \text{otherwise};
		\end{cases}$
		
		the $(v''_i,v''_j)$-th entry of $\mathcal{I}^\mathcal{L}(G)=\begin{cases}
			1-\frac{2l(v_i)}{d_i+d^-_i}, & \text{if~} i=j \text{~and~} d_i+d^{-}_i\neq 0;\\
			\frac{-\beta}{\sqrt{(d^{-}_{i}+d_{i})(d^{-}_{j}+d_{j})}}, & \text{if}~i\neq j\textnormal{ and}~ v_{i}\overset{\beta}{\sim} v_{j};\\
			0, & \text{otherwise};
		\end{cases}$
		
		$$\text{the }(v'_i,v''_j)\text{-th entry of } \mathcal{I}^\mathcal{L}(G)=\textnormal{the }(v''_j,v'_i)\text{-th entry of }\mathcal{I}^\mathcal{L}(G)=\begin{cases}
			\frac{-\gamma}{\sqrt{(d^{+}_{i}+d_{i})(d^{-}_{j}+d_{j})}}, & \text{if}~ v_{i}\overset{\gamma}{\rightarrow} v_{j};\\
			0, & \text{otherwise},
		\end{cases}$$
		
		\noindent where $d_i=d(v_i)$, $d^+_i=d^+(v_i)$, $d^-_i=d^-(v_i)$.
	\end{defn}
	Let $G$ be a mixed graph on $n$ vertices $v_1,v_2,\ldots,v_n$. Note that $G$ has no trivial mixed components if and only if, for each $i=1,2,\ldots,n$, at least one of the following holds:
	
	(i) $d(v_i)\neq 0$; or
	
	(ii) $d^+(v_i)\neq 0$ and $d^-(v_i)\neq 0$. 
	
	Since the diagonal entries of $\mathcal{I}^D(G)$ are $d(v_i)+d^+(v_i)$ and $d(v_i)+d^-(v_i)$ for $i=1,2,\ldots,n$, it follows that $G$ has no trivial mixed components if and only if $\mathcal{I}^D(G)$ is non-singular. Therefore, for a mixed graph $G$ having no trivial mixed components, by maintaining the same order in the indices of $\mathcal{I}^D(G)$ and $\mathcal{I}^L(G)$, the normalized integrated Laplacian matrix of $G$ can be written as:
	$$\mathcal{I}^\mathcal{L}(G)=\mathcal{I}^D(G)^{-\frac{1}{2}}\mathcal{I}^L(G)\mathcal{I}^D(G)^{-\frac{1}{2}}.$$
	
	Notice that each simple mixed graph can be determined from its normalized integrated Laplacian matrix. The eigenvalues of $\mathcal{I}^\mathcal{L}(G)$ are referred to as the $\mathcal{I}^\mathcal{L}$-\textit{eigenvalues} of $G$. These eigenvalues are denoted by $\hat{\bm{\nu}}_1(G)$ for $i=1,2,\ldots,2n$. Since $\mathcal{I}^\mathcal{L}(G)$  is a real symmetric matrix, its eigenvalues can be arranged, without loss of generality, in non-increasing order as: $\hat{\bm{\nu}}_1(G)\geq\hat{\bm{\nu}}_2(G)\geq\cdots\geq\hat{\bm{\nu}}_{2n}(G)$. The spectrum of $\mathcal{I}^\mathcal{L}(G)$ is called the $\mathcal{I}^\mathcal{L}$-\textit{spectrum} of $G$. Observe that  $\mathcal{I}^\mathcal{L}(G)=\widehat{L}(G^A)$ and therefore,  the spectrum of $\mathcal{I}^\mathcal{L}(G)$ is identical to the spectrum of  $\widehat{L}(G^A)$. 
	
%	\begin{example}
%		The normalized integrated Laplacian matrix of the mixed graph $G$ shown in Figure~\ref{fig-mixed-exmpl}(i) is \\ $\mathcal{I}^\mathcal{L}(G)=\begin{bmatrix}
%			\frac{5}{7}&\frac{-1}{\sqrt{70}}&\frac{-1}{2\sqrt{14}}&\frac{-1}{2\sqrt{7}}&\frac{-1}{3\sqrt{7}}&\frac{-1}{3\sqrt{7}}&0&0\\
%			\frac{-1}{\sqrt{70}}&\frac{4}{5}&0&\frac{-1}{\sqrt{10}}&\frac{-2}{3\sqrt{10}}&\frac{-2}{3\sqrt{10}}&0&\frac{-1}{5\sqrt{2}}\\
%			\frac{-1}{2\sqrt{14}}&0&\frac{1}{2}&0&\frac{-1}{6\sqrt{2}}&\frac{-1}{6\sqrt{2}}&0&\frac{-1}{2\sqrt{10}}\\
%			\frac{-1}{2\sqrt{7}}&\frac{-1}{\sqrt{10}}&0&1&0&0&\frac{-1}{2\sqrt{6}}&0\\
%			\frac{-1}{3\sqrt{7}}&\frac{-2}{3\sqrt{10}}&\frac{-1}{6\sqrt{2}}&0&\frac{7}{9}&\frac{-1}{9}&\frac{-1}{3\sqrt{6}}&\frac{-1}{3\sqrt{5}}\\
%			\frac{-1}{3\sqrt{7}}&\frac{-2}{3\sqrt{10}}&\frac{-1}{6\sqrt{2}}&0&\frac{-1}{9}&\frac{7}{9}&0&\frac{-2}{3\sqrt{5}}\\
%			0&0&0&\frac{-1}{2\sqrt{6}}&\frac{-1}{3\sqrt{6}}&0&\frac{1}{3}&0\\
%			0&\frac{-1}{5\sqrt{2}}&\frac{-1}{2\sqrt{10}}&0&\frac{-1}{3\sqrt{5}}&\frac{-2}{3\sqrt{5}}&0&1
%		\end{bmatrix}$.
%	\end{example}
	
	\subsection{Results}
	\begin{observation}\label{normalized observation}
		Let $G$ be a mixed graph. Then we have the following.
		\begin{enumerate}[(1)]
			\item $\mathcal{I}^\mathcal{L}(G)$ is positive semi-definite, since $\widehat{L}(G^A)$ is positive semi-definite.
			\item The least eigenvalue of $\mathcal{I}^\mathcal{L}(G)$ is $0$, since the least eigenvalue of $\widehat{L}(G^A)$ is $0$.
			%	\item The matrix $\mathcal{I}^\mathcal{L}(G)$ has $2n$ pairwise orthogonal eigenvectors in $\mathbb{R}^{2n}$ and so $\mathcal{I}^\mathcal{L}(G)$ is orthogonally similar to a diagonal matrix with real entries.
			%	\item If $d(u)+d^+(u)=0$ or $d(u)+d^-(u)=0$, then $0$ is an eigenvalue of $\mathcal{I}^\mathcal{L}(G)$ and hence $|\mathcal{I}^\mathcal{L}(G)|=0$. In particular, if $u$ is isolated, then $0$ is an eigenvalue of $\mathcal{I}^\mathcal{L}(G)$ with multiplicity at least $2$.
			%	\item If $0$ is an eigenvalue of $\mathcal{I}^\mathcal{L}(G)$, then its multiplicity is at least 
			%	$2|\{u\in V_G|d(u)=d^+(u)=d^-(u)=0\}|+|\{u\in V_G|d(u)=d^+(u)=0 \text{~and~} d^-(u)\neq0\}|+|\{u\in V_G|d(u)=d^-(u)=0 \text{~and~} d^+(u)\neq0\}|$.
			\item If $G$ is a graph, then $\mathcal{I}^\mathcal{L}(G)=I_2\otimes\widehat{L}(G)$. The eigenvalues of $\mathcal{I}^\mathcal{L}(G)$ are the same as those of $\widehat{L}(G)$, with each eigenvalue having twice its algebraic multiplicity.
			%	\item Let $G_1,G_2,\ldots,G_k$ be $k$ pairwise disjoint mixed graphs. Then the characteristic polynomial of $\mathcal{I}^\mathcal{L}(\bigcup_{i=1}^{k}G_i)$ is the product of characteristic polynomial of $\mathcal{I}^\mathcal{L}(G_i)$ for $i=1,2,\ldots,k$.
			\item If $G$ is $r$-regular, then $0$ is an eigenvalue of $\mathcal{I}^\mathcal{L}(G)$, with corresponding eigenvector $\boldsymbol{1}$.
			%	\item If $\begin{bmatrix}
				%		\boldsymbol{1}_n\\
				%		-\boldsymbol{1}_n
				%	\end{bmatrix}$ is an eigenvector of $\widehat{L}(G)$ with corresponding eigenvalue $s$, then $\displaystyle d(u)=\frac{s}{2}$.
			%	\item For each mixed graph $G$, since $\mathcal{I}^\mathcal{L}(G)=\widehat{L}(G^A)$ and $\widehat{L}(G^A)$ is positive semi-definite, $\mathcal{I}^\mathcal{L}(G)$ is positive semi-definite.
		\end{enumerate}
	\end{observation}

	\begin{thm}\label{Th4.3}
		Let $G$ be a loopless plain mixed graph on $n$ vertices. Then the following statements hold:
		\begin{enumerate}[(i)]
			\item $\sum_{i=1}^{2n}\hat{\bm{\nu}}_i(G)\leq2n$, with equality if and only if $G$ has no trivial mixed component.
			\item $\sum_{i=1}^{2n}\hat{\bm{\nu}}_i(G)=2n-k$, where $k$ is the number of trivial mixed components of $G$.
			\item If $G$ is not directed loop complete, then $\hat{\bm{\nu}}_{2n-1}(G)\leq 1$.
			\item If $G$ has no trivial mixed component, then $\hat{\bm{\nu}}_{2n-1}(G)\leq\frac{2n}{2n-1}$, with equality occurring if and only if $G$ is directed loop complete.
			
			\item If $G$ has no trivial mixed component, then $\hat{\bm{\nu}}_{1}(G)\geq\frac{2n}{2n-1}$, with equality is achieved if and only if $G$ is directed loop complete.
			\item $\hat{\bm{\nu}}_1(G)\leq 2$, with equality if and only if $G$ has a non-trivial mixed component, which has the AB property. 
			In this case, the multiplicity of $\hat{\bm{\nu}}_1(G)= 2$ equals the number of non-trivial mixed components of $G$ having the AB property.
		\end{enumerate}
	\end{thm}
	\begin{proof}
		Notice that $G$ has no trivial mixed component if and only if $G^A$ has no isolated vertices. From \cite[Theorem~7.7.2(i)]{cvetkovic}, we have $\sum_{i=1}^{2n}\hat{\nu}_i(G^A)\leq2n$ with equality holds if and only if $G^A$ has no isolated vertices. Since $\mathcal{I}^\mathcal{L}(G)$ and $\widehat{L}(G^A)$ have the same spectrum, part~(i) follows. 
		
		To prove (ii), let $G$ has $k\geq 0$ trivial mixed components. This implies that $G^A$ has $k$ isolated vertices. By removing these $k$ isolated vertices from $G^A$, we obtain a subgraph $G'$ of $G^A$, having $2n-k$ vertices. From \cite[Theorem~7.7.2(i)]{cvetkovic}, we have $\sum_{i=1}^{2n-k}\hat{\nu}_i(G')=2n-k$. Also, $$\sum_{i=1}^{2n}\hat{\bm{\nu}}_i(G)=\sum_{i=1}^{2n}\hat{\nu}_i(G^A)=\sum_{i=1}^{2n-k}\hat{\nu}_i(G').$$ Hence part~(ii) follows. 
		
		From Lemma~\ref{complete} and \cite[parts (ii), (iii), (iv) of Theorem~7.7.2]{cvetkovic}, we have the following:
		\begin{itemize}
			\item  If $G^A$ is not complete, $\hat{\nu}_{2n-1}(G^A)\leq 1$.
			\item $\hat{\nu}_{2n-1}(G^A)\leq\frac{2n}{2n-1}$, with equality holds if and only if $G^A$ is complete.
			\item  $\hat{\nu}_1(G^A)\geq\frac{2n}{2n-1}$, with equality holds if and only if $G^A$ is complete. 	
		\end{itemize}
		 Since $\hat{\bm{\nu}}_{2n-1}(G)=\hat{\nu}_{2n-1}(G^A)$ and $\hat{\bm{\nu}}_{1}(G)=\hat{\nu}_{1}(G^A)$, parts (iii), (iv) and (v) follows.
		
		From \cite[Theorem~7.7.2(v)]{cvetkovic},
		$\hat{\nu}_1(G^A)\leq 2$, with equality is achieved if and only if $G^A$ has a non-trivial component which is bipartite. Since $\hat{\bm{\nu}}_1(G)=\hat{\nu}_1(G^A)$, from Lemma~\ref{ABcomponentbipartite}, first assertion of part (vi) follows.

		From \cite[Theorem~7.7.2(v)]{cvetkovic},  $2$ is a simple eigenvalue of $\widehat{L}(H)$, for a connected bipartite graph $H$. Therefore, if $G^A$ has a non-trivial bipartite component, then $2$ is an eigenvalue of $\widehat{L}(G^A)$, with multiplicity equal to the number of non-trivial bipartite components of $G^A$ which is the same as the number of non-trivial mixed components of $G$ having the AB property. This completes the proof of part~(vi).
	\end{proof}
	
	\begin{cor}
		If $G$ is a loopless plain bipartite mixed graph having at least one edge or arc, then $\hat{\bm{\nu}}_1(G)=2$ with multiplicity equal to the number of non-trivial mixed components of $G$.
	\end{cor}
	\begin{proof}
		If $G$ is a bipartite mixed graph with at least one edge or arc, then each non-trivial mixed component of $G$possesses the AB property. Consequently, the result follows directly from part~(vi) of Theorem~\ref{Th4.3}.
	\end{proof}
	\begin{thm}
	If $G$ is a loopless plain mixed graph on $n$ vertices, then
		$$\hat{\bm{\nu}}_{2n-1}(G)\leq\frac{2n-k}{2n-1}\leq\hat{\bm{\nu}}_{1}(G),$$ where $k$ is the number of trivial mixed components of $G$.
	\end{thm}
	\begin{proof}
	Given that  $\hat{\bm{\nu}}_{2n-1}(G)\leq\hat{\bm{\nu}}_i(G)\leq\hat{\bm{\nu}}_1(G)$ for $i=1,2,\ldots,2n-1$, we deduce $$(2n-1)\hat{\bm{\nu}}_{2n-1}(G)\leq\underset{i=1}{\overset{2n-1}{\sum}}\hat{\bm{\nu}}_i(G)\leq(2n-1)\hat{\bm{\nu}}_1(G).$$
		Since $\hat{\bm{\nu}}_{2n}(G)=0$, part (ii) of Theorem~\ref{Th4.3} implies that $\underset{i=1}{\overset{2n-1}{\sum}}\hat{\bm{\nu}}_i(G)=2n-k$. 
	Substituting this in the above inequality completes the proof.
	\end{proof}
	\begin{thm}\label{Th4.4}
		Let $ G $ be a simple mixed graph. Then, the number of mixed components in $ G $ is equal to the multiplicity of $ 0 $ as an eigenvalue of $ \mathcal{I}^\mathcal{L}(G) $.
	\end{thm}
	\begin{proof}
		Since the spectrum of $\mathcal{I}^\mathcal{L}(G)$ is identical to the spectrum of $\widehat{L}(G^A)$, the result follows from \cite[Corollary~6.1]{Kalaimatrices1}: ``The number of mixed components of a mixed graph $G$ is equal to the number of components of $G^A$'' and \cite[Theorem~7.7.3]{cvetkovic}. 
%		``For any simple graph $G$, the multiplicity of $0$ as an eigenvalue of $\widehat{L}(G)$ is equal to the number of components in $G$''.
	\end{proof}
	
	\begin{cor}
		Let $G$ be a simple mixed graph on $n$ vertices with  no trivial mixed components. Then $G$ has the AB property if and only if $\hat{\bm{\nu}}_1(G)=2$, with the same multiplicity as $\hat{\bm{\nu}}_{2n}(G)=0$.
	\end{cor}
	\begin{proof}
		Since $G$ has no trivial mixed components, it has the AB property if and only if each mixed component of $G$ has the AB property.
		So from Theorems~\ref{Th4.4} and \ref{Th4.3}(vi), the result follows.
	\end{proof}
	
	\begin{thm}\label{thm-normalized regular}
		Let $G$ be an $r$-regular simple mixed graph on $n$ vertices.
		\begin{itemize}
			\item[(i)] If $\bm{\lambda}_1(G)\geq\bm{\lambda}_{2}(G)\geq\cdots\geq\bm{\lambda}_{2n}(G)$ is the $\mathcal{I}$-spectrum of $G$  with corresponding eigenvectors $X_1,X_2,\ldots,X_{2n}$, then the $\mathcal{I}^\mathcal{L}$-spectrum of $G$ is $1-\frac{\bm{\lambda}_1(G)}{r}\leq 1-\frac{\bm{\lambda}_2(G)}{r}\leq \cdots\leq 1-\frac{\bm{\lambda}_{2n}(G)}{r}$  with corresponding eigenvectors $X_1,X_2,\ldots,X_{2n}$.
			\item[(ii)] If $\bm{\mu}_1(G)\geq\bm{\mu}_{2}(G)\geq\cdots\geq\bm{\mu}_{2n}(G)$ is the $\mathcal{I}^L$-spectrum of $G$  with corresponding eigenvectors $X_1,X_2,\ldots,X_{2n}$, then the $\mathcal{I}^\mathcal{L}$-spectrum of $G$ is $\frac{\bm{\mu}_1(G)}{r}\geq \frac{\bm{\mu}_2(G)}{r}\geq \cdots\geq \frac{\bm{\mu}_{2n}(G)}{r}$  with corresponding eigenvectors $X_1,X_2,\ldots,X_{2n}$.
			\item[(iii)] If $G$ is a directed graph such that $(u,v)\in\vec{E}_G$ if and only if $(v,u)\in\vec{E}_G$, and the spectrum of $\vec{A}(G)$ is $\lambda_1\geq\lambda_{2}\geq\cdots\geq\lambda_n$ with corresponding eigenvectors $X_1,X_2,\ldots,X_n$, then the $\mathcal{I}^\mathcal{L}$-spectrum of $G$ is $1-\frac{2\lambda_i}{r}$  with corresponding eigenvector $\begin{bmatrix}
				X_i^T&
				X_i^T
			\end{bmatrix}^T$ for $i=1,2,\ldots,n$, and $1^{(n)}$ with corresponding eigenvectors $\begin{bmatrix}
				X_i^T&
				-X_i^T
			\end{bmatrix}^T$  for $i=1,2,\ldots,n$.
			
			\item[(iv)] If $G$ is a graph and the spectrum of $A(G)$ is $\lambda_1(G)\geq\lambda_{2}(G)\geq\cdots\geq\lambda_n(G)$ with corresponding eigenvectors $X_1,X_2,\ldots,X_n$, then the $\mathcal{I}^\mathcal{L}$-spectrum of $G$ is $(1-\frac{\lambda_i(G)}{r})^{(2)}$  with  corresponding eigenvectors $\begin{bmatrix}
				X_i^T&
				\mathbf{0}_{1\times n}
			\end{bmatrix}^T$ and $\begin{bmatrix}
				\mathbf{0}_{1\times n}&
				X_i^T
			\end{bmatrix}^T$ for $i=1,2,\ldots,n$.
			
			\item[(v)] If $G$ is a graph and the spectrum of $L(G)$ is $\mu_1(G)\geq\mu_{2}(G)\geq\cdots\geq\mu_n(G)$ with corresponding eigenvectors $X_1,X_2,\ldots,X_n$, then the $\mathcal{I}^\mathcal{L}$-spectrum of $G$ is $(\frac{\mu_i(G)}{r})^{(2)}$ with corresponding eigenvectors $\begin{bmatrix}
				X_i^T&
				\mathbf{0}_{1\times n}
			\end{bmatrix}^T$ and $\begin{bmatrix}
				\mathbf{0}_{1\times n}&
				X_i^T
			\end{bmatrix}^T$ for $i=1,2,\ldots,n$.
		\end{itemize}
	\end{thm}
	\begin{proof}
		Since $G$ is $r$-regular, we have $\mathcal{I}^D(G)=rI_{2n}$, which implies $\mathcal{I}^D(G)^{\frac{1}{2}}=\displaystyle\frac{1}{\sqrt{r}}I_{2n}$. Therefore,  $\mathcal{I}^\mathcal{L}(G)=I_{2n}-\displaystyle\frac{1}{r}\mathcal{I}(G)$. Hence part~(i) follows. The remaining parts can be established in a similar manner.
	\end{proof}
	
	\begin{thm}\label{normalizedlaplaciancomplete}
		\begin{enumerate}[(i)]
			\item The $\mathcal{I}^\mathcal{L}$-spectrum of $K^M_{k(m)}$ is $0^{(1)}$, $(\frac{k}{k-1})^{(k-1)}$, $1^{(2mk-k)}$.
			
			\item The $\mathcal{I}^\mathcal{L}$-spectrum of $K^D_{k(m)}$ is $2^{(1)}$, $0^{(1)}$, $(\frac{k}{k-1})^{(k-1)}$, $(\frac{k-2}{k-1})^{(k-1)}$, and $1^{(2km-2k)}$.
		\end{enumerate}
	\end{thm}
	\begin{proof}
		From Theorem~\ref{thm-normalized regular}(i), the results can be proved similar to the proof of Theorem~\ref{laplaciancomplete}.
	\end{proof}
	
	\begin{cor}
		\begin{itemize}
			\item[(i)] The $\mathcal{I}^\mathcal{L}$-spectrum of $K^M_n$ is $0^{(1)}$, $(\frac{n}{n-1})^{(n-1)}$, $1^{(n)}$.
			
			\item[(ii)] The $\mathcal{I}^\mathcal{L}$-spectrum of $K^D_n$ is $0^{(2)}$, $(\frac{n}{n-1})^{(n-1)}$, $(\frac{n-2}{n-1})^{(n-1)}$.
		\end{itemize}
	\end{cor}
	\begin{proof}
		By taking $k=n$ and $m=1$ in Theorem~\ref{normalizedlaplaciancomplete}, the result follows.
	\end{proof}
	\begin{thm}
		\begin{enumerate}[(i)]
			\item If $G$ is the oriented graph of $P_n$ $(n\geq 2)$ with all its arcs have the same direction, then the $\mathcal{I}^\mathcal{L}$-spectrum of $G$ is $2^{(n-1)}$, $0^{(n+1)}$.
			
			\item If $G$ is the oriented graph of $C_n$ $(n\geq 3)$ with all its arcs have the same direction, then the $\mathcal{I}^\mathcal{L}$-spectrum of $G$ is $2^{(n)}$, $0^{(n)}$.
			
			\item If $G$ is the oriented graph of $C_n$ $(n\geq 4)$ with all its arcs have the alternating direction, then the $\mathcal{I}^\mathcal{L}$-spectrum of $G$ is $0^{(n)}$, and $(1-\cos\frac{2\pi k}{n})^{(1)}$ for $k=1,2,\ldots,n$.
		\end{enumerate}
	\end{thm}
	\begin{proof}
		The proof is similar to that of Theorem~\ref{orientedLaplacian}.
	\end{proof}
	\begin{thm}Let $G$ be a uniconnected mixed graph on $n$ vertices. If $G$ has an alternating cycle of length $2n$  that includes all  vertices, all  arcs and each edge of $G$ exactly twice, then the $\mathcal{I}^\mathcal{L}$-spectrum of $G$ is $1-\cos\frac{\pi k}{n}$ for $k=1,2,\ldots,2n$.
	\end{thm}
	\begin{proof}
		From Theorem~\ref{thm-normalized regular}(i), the result can be proved similar to the proof of Theorem~\ref{Laplacian spectrum cycle}.
	\end{proof}
	%A \textit{contraction} of a graph $G$ is formed by identifying two distinct vertices, say $u$ and $v$, into a single vertex $v^*$.
	\begin{thm}
		Let $G$ be a simple mixed graph, and let $u,v\in V_G$ be such that $u\neq v$ and $u\not\sim v$. If $H$ is obtained by identifying $u$ and $v$ into a single vertex, then $\hat{\bm{\nu}}_{2n-1}(G)\leq\hat{\bm{\nu}}_{2n-3}(H)$.
	\end{thm}
	\begin{proof}
		Let $u',u''$ and $v',v''$ be the vertices of $G^A$ corresponding to the vertices $u$ and $v$, respectively. Then $H^A$ is a contraction of $G^A$, which is obtained by identifying $u'$ and $v'$ into a single vertex, and $u''$ and $v''$ into a single vertex. This implies that $H^A$ has $2n-2$ vertices. From \cite[Lemma~1.15]{chung}: ``Let $G$ be a graph on $n$ vertices. If $H$ is obtained by contractions from $G$ and $H$ has $m$ vertices, then $\hat{\nu}_{n-1}(G)\leq\hat{\nu}_{m-1}(H)$'', it follows that $\hat{\nu}_{2n-1}(G^A)\leq\hat{\nu}_{2n-3}(H^A)$. Since $\hat{\bm{\nu}}_{2n-1}(G)=\hat{\nu}_{2n-1}(G^A)$ and $\hat{\bm{\nu}}_{2n-3}(H)=\hat{\nu}_{2n-3}(H^A)$, the result follows.
	\end{proof}
	\begin{defn}\normalfont
		Let $G$ be a mixed graph, and let $X,Y\subseteq V_G$. We define $$e(X,Y)=|\{\{u,v\}\in E_G\mid u\in X\text{ and }v\in Y\}|$$ and $$a(X,Y)=|\{(u,v)\in \vec{E}_G\mid u\in X\text{ and }v\in Y\}|.$$
		
		For $S\subseteq V_G$,  the \textit{mixed volume of $S$ in $G$} is defined as $$\text{vl}(S)=\sum_{v\in S}(2d(v)+d^+(v)+d^-(v)).$$
	\end{defn} 
	
	Note that in particular, for a graph $G$, we have $\text{vl}(S)=2\text{vol}(S)$.
	
	\begin{lemma}\label{volume}
		Let $G$ be a mixed graph, and let $X\subseteq V_G$. Then we have $\textnormal{vl}(X)=\textnormal{vol}(\bm{X})$, where $\bm{X}$ denote the set of two copies of the elements of $X$ in $G^A$.
	\end{lemma}
	\begin{proof}
		Let $V_G=\{v_1,v_2,\ldots,v_n\}$, and without loss of generality, let $X=\{v_1,v_2,\ldots,v_k\}$ where $k\leq n$. Then $\bm{X}=\{v'_1,v'_2,\ldots,v'_k,v''_1,v''_2,\ldots,v''_k\}$. the mixed volume of $\bm{X}$ given by $$\textnormal{vl}(X)=\overset{k}{\underset{i=1}{\sum}}(2d(v_i)+d^+(v_i)+d^-(v_i))$$ and the volume of $\bm{X}$ is $$\textnormal{vol}(\bm{X})=\underset{v\in \bm{X}}{\sum}d(v)=\overset{k}{\underset{i=1}{\sum}}(d(v'_i)+d(v''_i)).$$ Since $d(v'_i)=d(v_i)+d^+(v_i)$ and $d(v''_i)=d(v_i)+d^-(v_i)$ for $i=1,,2,\ldots,k$, the result follows. 
	\end{proof}
	Let $G$ be a mixed graph. For $X\subseteq V_G$, we denote the complement of $X$ in $G$ by $X^c=V_G\setminus X$.
	\begin{thm}\label{eXY}
		Let $G$ be a simple mixed graph on $n$ vertices, having at least one edge or one arc, and let $X,Y\subseteq V_G$. Then we have
		\begin{enumerate}[(i)]
			\item $\left|2e(X,Y)+a(X,Y)-\frac{\textnormal{vl}(X)\textnormal{vl}(Y)}{4e(G)+2a(G)}\right|\leq \hat{\bm{\nu}}(G)\sqrt{\textnormal{vl}(X)\textnormal{vl}(Y)}$;
			\item $\left|2e(X,Y)+a(X,Y)-\frac{\textnormal{vl}(X)\textnormal{vl}(Y)}{4e(G)+2a(G)}\right|\leq \hat{\bm{\nu}}(G)\frac{\sqrt{\textnormal{vl}(X)\textnormal{vl}(Y)\textnormal{vl}(X^c)\textnormal{vl}(Y^c)}}{4e(G)+2a(G)}$,
		\end{enumerate}  where $\hat{\bm{\nu}}(G)=\underset{i\neq 2n}{\max}|1-\hat{\bm{\nu}}_i(G)|$.  
	\end{thm}
	\begin{proof}
		Let $\bm{X}$ (resp. $\bm{Y}$) denote the set of two copies of the elements of $X$ (resp. $Y$) in $G^A$. 
		Then $\bm{X}^c$ (resp. $\bm{Y}^c$) is the set of two copies of the elements of $X^c$ (resp. $Y^c$) in $G^A$. Since $e(\bm{X},\bm{Y})=2e(X,Y)+a(X,Y)$,  $\textnormal{vol}(V_{G^A})=\textnormal{vl}(V_G)=4e(G)+2a(G)$ and $\hat{\bm{\nu}}_i(G)=\hat{\nu}_i(G)$ for $i=1,2,\ldots,2n$, result follows from Lemma~\ref{volume} and \cite[Theorems~5.1 and 5.2]{chung}. 
	\end{proof}
	As a direct consequence of Theorem~\ref{eXY}, we get the next result.
	\begin{cor}
		Let $G$ be a simple mixed graph on $n$ vertices, and let $X\subseteq V_G$. Then we have the following.
		\begin{enumerate}[(i)]
			\item $\left|2e(X,X)+a(X,X)-\frac{(\textnormal{vl}(X))^2}{4e(G)+2a(G)}\right|\leq \hat{\bm{\nu}}(G)\frac{\textnormal{vl}(X)\textnormal{vl}(X^c)}{4e(G)+2a(G)}\leq \hat{\bm{\nu}}(G)\textnormal{vl}(X)$;
			\item If $G$ is $r$-regular, then $\left|2e(X,X)+a(X,X)-\frac{r|X|^2}{n}\right|\leq 2r\hat{\bm{\nu}}(G)|X|$.
		\end{enumerate}
	\end{cor}
	\begin{defn}\normalfont
		Let $G$ be a uniconnected mixed graph. For each distinct pairs of vertices $u,v\in V_G$, we define the \textit{distance} between $u$ and $v$, denoted by $d_G(u,v)$ as $$d_G(u,v)=\min\{d_1(u,v),d_2(u,v),d_3(u,v),d_4(u,v)\},$$ and $d_G(u,u)=0$, where
		\begin{itemize}
			\item $d_1(u,v)$ denotes  the minimum length of a shortest path joining $u$ and $v$ in $G$ if such a path exists, and the length of a shortest alternating path from $u$ to $v$ that contains an  even number of arcs, with the first arc in the forward direction, provided such a path exists;
			\item $d_2(u,v)$ denotes the length of a shortest alternating path from $u$ to $v$ that contains an number of arcs, with  the first arc is in the forward direction;
			\item $d_3(u,v)$ denotes the length of a shortest alternating path from $u$ to $v$ that contains an odd number of arcs, with  the first arc is in the backward direction;
			\item $d_4(u,v)$ denotes the minimum of the length of a shortest path joining $u$ and $v$ in $G$ if it exists, and the length of a shortest alternating path from $u$ to $v$ that contains an even number of arcs and, with  the first arc is in the backward direction, provided such a path exists.
		\end{itemize}
	\end{defn}
	
	\begin{defn}\normalfont
		Let $G$ be a uniconnected mixed graph. 
		For $X,Y\subseteq V_G$, the \textit{distance between $X$ and $Y$}, denoted by $d(X,Y)$, is defined as $$d(X,Y)=\min\{d_G(u,v)\mid u\in X\textnormal{ and }v\in Y\}.$$
	\end{defn}
	\begin{lemma}\label{distancelemma}
		Let $G$ be a mixed graph, and let $X,Y\subseteq V_G$. Then we have $d(X,Y)=d(\bm{X},\bm{Y})$, where $\bm{X}$ (resp. $\bm{Y}$) denote the set of two copies of the elements of $X$ (resp. $Y$) in $G^A$.
	\end{lemma}
	\begin{proof}
		Suppose $X\cap Y\neq \emptyset$. Then there exists $v\in X\cap Y$. Since $d_G(v,v)=0$, we have $d(X,Y)=0$. moreover, $v',v''\in \bm{X}\cap \bm{Y}$, where $v'$ and $v''$ are two copies of $v$ in $G^A$. Since $d(v',v')=d(v'',v'')=0$, we have $d(\bm{X},\bm{Y})=0$. 
		
		Now, suppose $X\cap Y= \emptyset$. Then we have $\bm{X}\cap \bm{Y}= \emptyset$. We can write $d(X,Y)=\min A$, where $A=\{d_1(u,v),d_2(u,v),d_3(u,v),d_4(u,v)\mid u\in X\textnormal{ and }v\in Y\}$. Let $B=\{d(u,v)\mid u\in \bm{X}\textnormal{ and }v\in \bm{Y}\}$. Then $d(\bm{X},\bm{Y})=\min B$.  
		For $u\in X$ and $v\in Y$, we have the following relationships between the distances: $d_1(u,v)=d(u',v')$, $d_2(u,v)=d(u',v'')$, $d_3(u,v)=d(u'',v')$ and $d_4(u,v)=d(u'',v'')$, where $u',u''\in\bm{X}$ (resp. $v',v''\in\bm{Y}$) are two copies of $u$ (resp. $v$) in $G^A$. From this, it is clear that $A=B$. Therefore,  $\min A=\min B$. 	
	\end{proof}
	\begin{thm}
		Let $G$ be a simple uniconnected mixed graph on $n$ vertices. For $X,Y\subseteq V_G$, we have 
		$$d(X,Y)\leq \left\lceil\frac{\cosh^{-1} \sqrt{\frac{\textnormal{vl}(X^c)\textnormal{vl}(Y^c)}{\textnormal{vl}(X)\textnormal{vl}(Y)}}}{\cosh^{-1} \frac{\hat{\bm{\nu}}_{1}(G)+\hat{\bm{\nu}}_{2n-1}(G)}{\hat{\bm{\nu}}_{1}(G)-\hat{\bm{\nu}}_{2n-1}(G)}}\right\rceil.$$
	\end{thm}
	\begin{proof}
		Let $\bm{X}$ (resp. $\bm{Y}$) denote the set of two copies of the elements of $X$ (resp. $Y$) in $G^A$. 
		Then we have $\bm{X}^c$ (resp. $\bm{Y}^c$) is the set of two copies of the elements of $X^c$ (resp. $Y^c$) in $G^A$.
		Since $\hat{\bm{\nu}}_1(G)=\hat{\nu}_1(G)$ and $\hat{\bm{\nu}}_{2n-1}(G)=\hat{\nu}_{2n-1}(G)$, result follows from Lemmas~\ref{volume}, \ref{distancelemma} and	\cite[Theorem~3.3]{chung}.
	\end{proof}
	
	\begin{thm}\label{distance}
		Let $G$ be a simple, uniconnected mixed graph on $n$ vertices. For $X_i\subset V_G, i=1,2,\ldots,k$ we have 
		\begin{enumerate}[(i)]
			\item $\underset{i\neq j}{\min}~d(X_i,X_j)\leq \underset{i\neq j}{\max}\left\lceil\frac{\log \sqrt{\frac{\textnormal{vl}X_i^c\textnormal{vl}X_j^c}{\textnormal{vl}X_i\textnormal{vl}X_j}}}{\log \frac{1}{1-\hat{\bm{\nu}}_{2n-k+1}(G)}}\right\rceil$ if $1-\hat{\bm{\nu}}_{2n-k+1}(G)\geq \hat{\bm{\nu}}_{1}(G)-1$;
			\item $\underset{i\neq j}{\min}~d(X_i,X_j)\leq \underset{i\neq j}{\max}\left\lceil\frac{\log \sqrt{\frac{\textnormal{vl}X_i^c\textnormal{vl}X_j^c}{\textnormal{vl}X_i\textnormal{vl}X_j}}}{\log \frac{\hat{\bm{\nu}}_{1}(G)+\hat{\bm{\nu}}_{2n-k+1}(G)}{\hat{\bm{\nu}}_{1}(G)-\hat{\bm{\nu}}_{2n-k+1}(G)}}\right\rceil$ if $\hat{\bm{\nu}}_{2n-k+1}(G)\neq \hat{\bm{\nu}}_{1}(G)$;
			\item $\underset{i\neq j}{\min}~d(X_i,X_j)\leq \underset{1\leq j<k}{\min}\underset{i\neq j}{\max}\left\lceil\frac{\log \sqrt{\frac{\textnormal{vl}X_i^c\textnormal{vl}X_j^c}{\textnormal{vl}X_i\textnormal{vl}X_j}}}{\log \frac{\hat{\bm{\nu}}_{j+1}(G)+\hat{\bm{\nu}}_{2n-k+j-1}(G)}{\hat{\bm{\nu}}_{j+1}(G)-\hat{\bm{\nu}}_{2n-k+j-1}(G)}}\right\rceil$ if $\hat{\bm{\nu}}_{2n-k+j-1}(G)\neq \hat{\bm{\nu}}_{j+1}(G)$.
		\end{enumerate}
	\end{thm}
	\begin{proof}
		For $i=1,2,\ldots,k$, let $\bm{X_i}$ denote the set of two copies of the elements of $X_i$ in $G^A$. 
		Then $\bm{X_i}^c$ is the set of two copies of the elements of $X_i^c$ in $G^A$. From Lemma~\ref{distancelemma}, we have
		$d(X_i,X_j)=d(\bm{X_i},\bm{X_j})$ for $i,j\in\{1,2,\ldots,k\}$. Since $\hat{\bm{\nu}}_i(G)=\hat{\nu}_i(G)$ for $i=1,2,\ldots,2n$, result follows from Lemma~\ref{volume} and \cite[Theorems~3.10, 3.11, 3.12]{chung}.
	\end{proof}
	By taking $k=2$ in part~(ii) of Theorem~\ref{distance}, we get the next result.
	\begin{cor}
		Let $G$ be a simple uniconnected mixed graph on $n$ vertices. For $X,Y\subset V_G$, we have  
		$$d(X,Y)\leq \left\lceil\frac{\log \sqrt{\frac{\textnormal{vl}(X^c)\textnormal{vl}(Y^c)}{\textnormal{vl}(X)\textnormal{vl}(Y)}}}{\log \frac{\hat{\bm{\nu}}_{1}(G)+\hat{\bm{\nu}}_{2n-1}(G)}{\hat{\bm{\nu}}_{1}(G)-\hat{\bm{\nu}}_{2n-1}(G)}}\right\rceil.$$
	\end{cor}
\section{Concluding remarks}
It is shown in this paper that the integrated Laplacian matrix, integrated signless Laplacian matrix and
normalized integrated Laplacian matrix associated for a mixed graph are respectively identical to the Laplacian matrix, signless Laplacian matrix and normalized Laplacian matrix of the associated graph. This allowed us to use the spectra of these matrices as a means to connect the structural properties of the mixed graph with those of the associated graph.

Some spectral graph theoretic results for simple graphs, utilized in this paper, can be extended to graphs containing loops and/or multiple edges. Consequently, the results presented here, initially proved for simple mixed graphs using these spectral graph theoretic results for simple graphs, can be generalized to mixed graphs with multiple loops, multiple directed loops, multiple edges, and/or multiple arcs.
		
\end{document}